\documentclass[12pt]{amsart}

\usepackage{fullpage}
\usepackage{enumerate,mathdots,mathrsfs,verbatim,amssymb}
\usepackage{color,url}

\numberwithin{equation}{section}

\newtheorem{prop}{Proposition}[section]
\newtheorem{theorem}[prop]{Theorem}
\newtheorem{cor}[prop]{Corollary}
\newtheorem{lemma}[prop]{Lemma}

\theoremstyle{definition}
\newtheorem{defn}[prop]{Definition}
\newtheorem{example}[prop]{Example}

\newtheorem{assume}[prop]{Assumption}
\theoremstyle{remark}
\newtheorem{rem}[prop]{Remark}

\newcommand{\A}{\mathscr{A}}

\newcommand{\C}{\mathbb{C}}
\newcommand{\D}{\mathcal{D}}
\newcommand{\F}{\mathcal{F}}

\renewcommand{\H}{\mathcal{H}}

\newcommand{\K}{\mathcal{K}}
\newcommand{\M}{\mathcal{M}}
\newcommand{\N}{\mathbb{N}}

\newcommand{\Q}{\mathbb{Q}}
\newcommand{\R}{\mathbb{R}}

\newcommand{\T}{\mathbb{T}}
\newcommand{\Z}{\mathbb{Z}}

\newcommand{\ve}{\varepsilon}
\newcommand{\ov}{\overline}

\newcommand{\Norm}[1]{\left\Vert #1 \right\Vert}

\newcommand{\caret}{\char`\^}

\renewcommand{\subset}{\subseteq}
\renewcommand{\supset}{\supseteq}

\DeclareMathOperator{\esssup}{ess\,sup}
\DeclareMathOperator{\essran}{ess\,ran}

\DeclareMathOperator{\spn}{span}
\DeclareMathOperator{\ran}{ran}
\DeclareMathOperator{\rank}{rank}

\begin{document}

\title{Multiplication-invariant operators and the classification of LCA group frames}
\author{Marcin Bownik}
\address{Department of Mathematics, University of Oregon, Eugene, OR 97403--1222, USA}
\email{mbownik@uoregon.edu}

\author{Joseph W. Iverson}
\address{Department of Mathematics, Iowa State University, Ames, IA 50011--2064, USA}
\email{jwi@iastate.edu}
\date{\today}

\thanks{ The first author was partially supported by NSF grant
  DMS-1665056. The second author was partially supported by ARO W911NF-16-1-0008.}

\subjclass{Primary: 42C15, 43A32,  47A15; Secondary: 43A65, 46C05}

\keywords{shift-invariant space, multiplication-invariant space, range function, range operator, dimension function, LCA group frame}

\begin{abstract} In this paper we study the properties of multiplication invariant (MI) operators acting on subspaces of the vector-valued space $L^2(X;\mathcal H)$. We characterize such operators in terms of range functions by showing that there is an isomorphism between the category of MI spaces (with MI operators as morphisms) and the category of measurable range functions whose morphisms are measurable range operators. We investigate how global properties of an MI operator are reflected by local pointwise properties of its corresponding range operator. We also establish several results about frames  generated by multiplications in $L^2(X;\mathcal H)$. This includes the classification of frames of multiplications with respect to unitary equivalence by measurable fields of Gramians. Finally, we show applications of our results in the study of abelian group frames and translation-invariant (TI) operators acting on subspaces of $L^2(\mathcal G)$, where $\mathcal G$ is a locally compact group. 
\end{abstract}

\maketitle

\section{Introduction}

The main goal of this paper is to develop the theory of multiplication-invariant (MI) operators and study its applications for abelian group frames and translation-invariant spaces. A general theory of MI subspaces of a vector-valued space $L^2(X;\mathcal H)$, where $X$ is a measure space and $\mathcal H$ is a separable Hilbert space, was developed by the first author and Ross \cite{BR}. The development of MI spaces extracts a measure-theoretic component from a characterization of shift-invariant (SI) spaces in terms of range functions which dates back to the work of Helson \cite{He}. This result has been extended in various settings such as $\R^n$ \cite{B, BDR2, BDR, RS}, locally compact abelian (LCA) groups \cite {CP, KR}, and more general translation-invariant spaces \cite{BHP2, BR, I}. It is worth adding that in the setting of MI spaces the role of characters of an LCA group is played by the concept of determining set---a collection of functions in $L^\infty(X)$ whose linear span is dense in the weak-* topology.

However, much less attention was devoted to the study of operators acting between SI spaces which are invariant under shifts. In the setting of $\R^n$ the study of shift-preserving operators was initiated by the first author \cite{B} who has shown a characterization of SI operators in terms of range operators. In the LCA setting SI operators were studied in \cite{KR2}. In this paper we extend this characterization to the general setting of MI operators by appropriately identifying range operators with decomposable operators appearing in the setting of von Neumann algebras \cite{D2}. In particular, we show that the category of MI spaces (with MI operators as morphisms) is isomorphic with the category of measurable range functions whose morphisms are measurable range operators. Moreover, we show that this isomorphism preserves a large number of global properties of an MI operator $T$ in terms of local pointwise properties of its corresponding range operator $R$, which are satisfied uniformly almost everywhere over a measure space $X$. This also includes spectral properties between $T$ and $R$, which necessitates the use of the machinery of measurable set-valued maps \cite{AF, CV}. As an example, we show that an MI operator $T=\int_X^{\oplus} R(x) d\mu(x)$ is normal with spectrum contained in a compact set $K\subset \C$ if and only if the range operators $R(x)$ are normal with spectra contained in $K$ for a.e. $x\in X$. Moreover, using functional calculus for MI operators we exhibit the precise relationship between spectral measure of the global MI operator $T$ and spectral measures of (local) range operators $R(x)$.

The theory of MI spaces and MI operators is developed with the aim of applications in the study of Bessel systems generated by multiplications in $L^2(X;\mathcal H)$. Such study was initiated by the second author in \cite{I} with the introduction of the concept of a Parseval determining set in $L^1(X)$, which is a measure-theoretic generalization of characters on an LCA group. In particular, it was shown in \cite{I} that the global frame properties of an MI system are characterized in terms of local frame properties, which are satisfied uniformly almost everywhere over $X$, thus generalizing the characterization of SI frames over fibers from the work of the first author \cite{B}. Moreover, we show that the analogy between the global and local frame behavior goes much deeper in the setting of MI spaces. The analysis, synthesis, and frame operators of a Bessel system of multiplications, which themselves are MI operators, correspond to appropriate range operators acting locally over fibers indexed by $x\in X$. Furthermore, several natural properties such as orthogonality of ranges for a pair of Bessel systems, duality, and canonical duality for a pair of frames are also preserved.

In addition, we give a classification for Bessel systems of multiplications with respect to unitary equivalence in terms of positive, integrable MI operators with the corresponding range operator being the measurable field of Gramians $\{\mathrm{Gr}(x)\}_{x\in X}$. 
The same classification holds for frames of multiplications with the additional constraint that MI operators are locally invertible operators. In particular, we deduce that Parseval frames of multiplications are in one-to-one correspondence with range operators being orthogonal projections, hence with measurable range functions.

We show two major applications of our measure-theoretic results on MI operators and frames of multiplications. The first deals with admissible unitary representations $\pi: G \to U(\mathcal H)$ of an LCA group $G$, which are also known as square integrable representations \cite{TW06, R04, W08}. We show that every abelian group system, which is generated by a countable collection in $\mathcal H$ by the action of $\pi$, is unitarily equivalent to an appropriate system of multiplications (by characters of $G$) in the MI space $L^2(\hat G; \mathcal K)$, where $\mathcal K$ is a separable Hilbert space. This enables us to apply the above mentioned results for MI operators and Bessel systems of multiplications to characterize abelian group frames. Similar results have been widely reported in the setting of translation-invariant spaces on abelian groups~\cite{BHP2,BL93,B,BR,CP,I,KR,RS}, as well as for representations of finite groups~\cite{VW2,VW3}, discrete groups~\cite{BHP,TW06}, and compact groups~\cite{I2}. 
As a consequence, we obtain a classification (up to unitary equivalence) of abelian group frames with $N\in \N \cup \{\infty\}$ generators  as positive, locally invertible, and integrable MI operators on $L^2(\hat G; \ell^2_N)$. 

Our second application deals with translation-invariant (TI) subspaces of $L^2(\mathcal G)$, where the underlying locally compact group $\mathcal G$ does not have to be abelian. Instead, we merely assume that a subgroup of translations $G \subset \mathcal G$ is abelian. The second author \cite{I} has shown that the study of TI spaces at such level of generality can be reduced to the study of MI spaces using a generalized Zak transform which converts left translation operators into multiplication operators acting on appropriate vector-valued MI spaces $L^2(X;\mathcal H)$. Similar results were obtained by Barbieri, Hern\'andez, and Paternostro in the LCA setting \cite{BHP2} and for discrete non-commutative groups \cite{BHP, BHP3}. Using the generalized Zak transform, whose existence and properties were established in \cite{I}, we apply the above mentioned results on MI operators and frames of multiplications to the setting of TI spaces and frames of translates. In particular, we classify the unitary equivalence of TI spaces using the dimension function, thus generalizing the corresponding results for SI spaces on $\R^n$ shown in \cite{B}.

\section{Shift-invariant spaces and Bessel systems generated by group actions}
\label{sec: 2}

We begin with a piece of motivation. Consider a separable Hilbert space $\H$ evolving discretely over time through the action of a unitary $U$, and fix a countable collection of ``sensors'' $\A = \{u_i\}_{i\in I}$ in $\H$. Suppose we use these sensors to measure the evolutions of various $v \in \H$, thereby obtaining the data
\[ Tv = \{ \langle U^k v, u_i \rangle \}_{k \in \Z, i \in I} =  \{ \langle v, U^{-k} u_i \rangle \}_{k \in \Z, i \in I}. \]
In general it is possible to stably reconstruct any $v\in \H$ from $Tv$ if and only if the system $E(\A) := \{ U^k u_i \}_{k\in \Z, i\in I}$ is a frame for $\H$~\cite{ACCMP,ACMT17}. One of the goals of this paper is to classify such systems.

From a group-theoretic perspective, the unitary $U$ is equivalent to a unitary representation of $\Z$, and the system $E(\A)$ consists of orbits of $\A$ under this action. In this section, we consider reproducing systems of this form for a broad class of groups in place of $\Z$. In the general case we establish a close correspondence with the theory of translation-invariant spaces. For abelian groups we show that the action may as well be given by modulation on a multiplication-invariant space involving the Pontryagin dual. This puts a classification of all such reproducing systems within reach---provided we first understand morphisms of MI spaces. Following a detailed study of those morphisms in the body of this paper, we return to classify Bessel systems generated by abelian groups in Section~\ref{sec:LCAFrm}.

\subsection{Preliminaries}
We refer the reader to~\cite{F,HR1,HR2} for background on abstract harmonic analysis. Throughout Section~\ref{sec: 2} we fix a second-countable, locally compact group $G$, which is not necessarily abelian. By a (unitary) \emph{representation} of $G$, we mean a homomorphism $\pi \colon G \to U(\H)$ into the group of unitaries of a Hilbert space $\H$, continuous with the respect to the strong (or equivalently, the weak) operator topology on $U(\H)$. If $\pi' \colon G \to U(\H')$ is another such representation, then an operator $T \colon \H \to \H'$ is said to \emph{intertwine} $\pi$ and $\pi'$ when $T\pi(x) = \pi'(x)T$ for all $x \in G$.

We write $C_c(G)$ for the space of continuous complex-valued functions on $G$ having compact support. The \emph{left Haar measure} on $G$, denoted $d\mu_G(x)$ (or simply $dx$ when the meaning is clear), is a regular Borel measure such that
\begin{equation}
\label{eq:Haar}
\int_G f(x)\, dx = \int_G f(yx)\, dx \qquad \text{for all $y \in G$}
\end{equation}
whenever $f \in C_c(G)$. This measure is unique up to a scaling factor, which we now fix. Whenever we treat $G$ as a measure space below, it is with this measure in mind. Thus we abbreviate $L^p(G) = L^p(G,\mu_G)$ for $1 \leq p \leq \infty$. Likewise, if $I$ is any indexing set, we give it the counting measure and then talk freely of the product measure space $G\times I$.

To any representation $\pi \colon G \to U(\H)$ we associate the mapping $\pi\colon L^1(G) \to B(\H)$ defined weakly by the relation
\[ \langle \pi(f) u, v \rangle = \int_G f(x) \langle \pi(x) u, v \rangle\, dx \qquad (u,v \in \H). \]
A closed subspace $V \subset \H$ is invariant under the mappings $\{ \pi(x) : x \in G\}$ if and only if it is invariant under $\{ \pi(f) : f \in C_c(G)\}$. In particular, for any set of vectors $\{ u_i \}_{i\in I'}$ in $\H$ we have
\[ \overline{\spn}\{ \pi(x) u_i : i \in I',\, x \in G\} = \overline{\spn}\{ \pi(f) u_i : i \in I',\, f \in C_c(G) \}. \]

Let $\H$ be a separable Hilbert space, and let $(X,\mu)$ be a $\sigma$-finite measure space for which $L^2(X)$ is separable. We define $L^2(X;\H)=L^2(X,\mu;\H)$ to be the space of equivalence classes of measurable functions $\varphi \colon X \to \H$, modulo equality a.e., with the property that $\int_X \Norm{\varphi(x)}^2\, d\mu(x) < \infty$. It becomes a Hilbert space with the inner product
\[ \langle \varphi, \psi \rangle = \int_X \langle \varphi(x), \psi(x) \rangle\, d\mu(x) \qquad (\varphi,\psi \in L^2(X;\H)). \]
When $\H = L^2(Y)$ there is a natural identification $L^2(X\times Y) \cong L^2(X;L^2(Y))$ that maps $f \in L^2(X\times Y)$ to $\varphi\in L^2(X;L^2(Y))$ given by $\varphi(x)(y) = f(x,y)$ a.e.\ $x\in X$, $y \in Y$.

Taking $X = G$, we associate each $x \in G$ with the \emph{left translation} operator $L_x \colon L^2(G;\H) \to L^2(G;\H)$ given by $(L_x \varphi)(y) = \varphi(x^{-1} y)$. Then $x \mapsto L_x$ defines a unitary representation of $G$. When $\H = L^2(Y)$ as above, the identification $L^2(X;L^2(Y)) \cong L^2(X\times Y)$ gives the translation operators $L_x \in U( L^2(G \times Y))$ the form $(L_x f)(y,z) = f(x^{-1}y,z)$ for $f \in L^2(G \times Y)$, $x,y \in G$, and $z \in Y$.

\begin{defn} \label{defn:GSys}
Given a representation $\pi \colon G \to U(\H)$ and a sequence $\A = \{ u_i \}_{i\in I}$ in $\H$, we write $E(\A) = \{ \pi(x) u_i \}_{x\in G,\, i \in I}$. (This should properly be viewed as a map $G\times I \to \H$.) We say that $\A$ \emph{generates} $E(\A)$, so that the latter has \emph{$|I|$ generators}. We call $E(\A)$:
\begin{itemize}
\item
\emph{complete}, if $\overline{\spn}\{ \pi(x) u_i : x \in G,\, i \in I\} = \H$;
\item
a \emph{Bessel $G$-system}, if there is a constant $B>0$ such that
\[ \sum_{i\in I} \int_G | \langle v, \pi(x) u_i \rangle |^2\, dx \leq B \Norm{v}^2 \qquad \text{for all }v \in \H; \]
\item
a \emph{$G$-frame}, if there are constants $A,B>0$ such that
\[ A \Norm{v}^2 \leq \sum_{i\in I} \int_G | \langle v, \pi(x) u_i \rangle |^2\, dx \leq B \Norm{v}^2 \qquad \text{for all }v \in \H. \]
\end{itemize}
When $E(\A)$ is Bessel, its \emph{analysis operator} is the bounded linear map $T\colon \H \to L^2(G \times I)$ given by 
\begin{equation} \label{eq:analysis}
(Tv)(x,i) = \langle v, \pi(x) u_i \rangle \qquad (v \in \H,\, x \in G,\, i \in I).
\end{equation}
and its \emph{synthesis operator} is $T^* \colon L^2(G\times I) \to \H$, given by
\[ T^*f = \sum_{i \in I} \int_G f(x,i) \pi(x) u_i \, dx \qquad (f \in L^2(G\times I)), \]
with the vector-valued integral interpreted in the weak sense.
\end{defn}

What we have called Bessel $G$-systems (respectively, $G$-frames) are examples of continuous Bessel systems (respectively, continuous frames), in the sense of~\cite{AAG,K}. The present terminology emphasizes the role of $G$ in the construction of the system. We will discuss more general Bessel systems and frames (absent a group) in a later section.

\begin{defn}
Let $\pi \colon G \to U(\H)$ and $\pi' \colon G \to U(\H')$ be representations, and let $\A = \{ u_i \}_{i\in I}$ and $\A' = \{ u_i' \}_{i\in I}$ be sequences in $\H$ and $\H'$, respectively, with a common indexing set $I$. We say that $E(\A)$ and $E(\A')$ are \emph{unitarily equivalent} if there is a unitary
\[U \colon \overline{\spn}\{ \pi(x) u_i : i\in I,\, x \in G\} \to \overline{\spn}\{ \pi'(x) u_i' : i\in I,\, x \in G\} \]
satisfying $U\pi(x) u_i = \pi'(x) u_i'$ for every $i \in I$ and $x \in G$.
\end{defn}

We emphasize that unitary equivalence must preserve the indexing while mapping $E(\A)$ to $E(\A')$.
In case $E(\A)$ and $E(\A')$ are both complete, it is easy to show that a unitary equivalence between them is the same as a unitary $U \colon \H \to \H'$ that intertwines $\pi$ with $\pi'$ while mapping $U u_i = u_i'$ for every $i \in I$.
Finally, notice that ``Bessel $G$-system'' and ``$G$-frame'' are properties of unitary equivalence classes, while ``complete'' is not.
This concludes our review of the preliminaries.

\subsection{Reduction to translation}

\begin{assume}
For the remainder of Section~\ref{sec: 2} we use the following assumptions.
As above, $G$ denotes a second-countable locally compact group, which is not necessarily abelian.
In addition, we fix a representation $\pi \colon G \to U(\mathcal{H})$ with $\mathcal{H}$ separable.
We assume that $\pi$ is \emph{admissible}, in the sense that there is a countable sequence $\A = \{ u_i \}_{i\in I}$ in $\H$ for which $E(\A) = \{ \pi(x) u_i \}_{x\in G,\, i \in I}$ is both complete and Bessel.
We also fix a choice of $\A$.
\end{assume}

\begin{theorem} \label{thm: SI embedding}
Let $\mathcal{G}$ be any second-countable locally compact group containing $G$ as a closed subgroup of index $[\mathcal{G}:G] \geq |I|$. Then there is a closed subspace $V\subset L^2(\mathcal{G})$ invariant under left translation by $G$, and a unitary $U\colon \H \to V$ that intertwines $\pi$ with left translation. Consequently, $E(\A)$ is unitarily equivalent to the system of translates $\{ L_x U u_i \}_{x \in G, i \in I}$ in $V$.
\end{theorem}

When $|I| = 1$ and $\mathcal{G} = G$, it is well known that any Bessel system of the form $\{ \pi(x) u \}_{x\in G}$ may as well be given by left translation in $L^2(G)$. Theorem~\ref{thm: SI embedding} generalizes to the case of countably many generators. For compact groups, a version of this theorem appeared as \cite[Theorem~4.6]{I2}. 

When $G=\Z$, as in the introduction to this section, we can take $\mathcal{G} = \R$ in Theorem~\ref{thm: SI embedding}. Thus, any complete Bessel $\Z$-system $\{ U^k u_i \}_{k\in \Z, i \in I}$ may as well be given by integer shifts in a shift-invariant subspace of $L^2(\R)$. There is a large literature devoted to the study of such systems, which are closely related to multiresolution analysis in classical wavelet theory~\cite{BL93,BW94,B,BDR2,BDR,RS}.

The lemma below is essentially contained in a paper by Weber~\cite{W08}, following a strategy of Rieffel~\cite{R04}. We give a proof for the sake of clarity.

\begin{lemma}
\label{lem:Bessel regular rep}
Let $T \colon \H \to L^2(G\times I)$ be the analysis operator of \eqref{eq:analysis}. In its polar decomposition $T = UP$, the partial isometry $U \colon \H \to L^2(G\times I)$ is a linear isometry onto $\overline{\ran T}$ that intertwines $\pi$ with left translation. Thus, $E(\A)$ is unitarily equivalent to the system of translates $\{ L_x U u_i \}_{x\in G, i\in I}$ in $\overline{\ran T}$.
\end{lemma}

\begin{proof}
Given any $y \in G$ and $v \in \H$, we have
\[ (T \pi(y) v)(x,i) = \langle \pi(y) v, \pi(x) u_i \rangle = \langle v, \pi(y^{-1} x) u_i \rangle = (Tv)(y^{-1} x,i) = (L_y Tv)(x,i) \]
for all $x \in G$ and $i \in I$. Consequently, $T \pi(y) = L_y T$ for all $y \in G$. It follows that $U$ intertwines $\pi$ with the regular representation; see \cite[Prop.~VI.13.13]{FD88}.

Now we only need to prove that $U$ is an isometry. Since $\ker U = (\ran P)^\perp = \ker T$, it suffices to prove that the synthesis operator $T^*$ has dense range in $\H$. Given an index $j \in I$ and a function $f\in C_c(G)$, we let $f_j \in L^2(G\times I)$ be the function 
\[ f_j(x,i) = \begin{cases} f(x), & \text{if }i = j; \\ 0, & \text{otherwise.} \end{cases} \]
Then
\[ T^* f_j = \int_G f(x) \pi(x) u_j\ dx = \pi(f) u_j. \]
Holding $j$ fixed and letting $f$ vary across $C_c(G)$, we see that $\overline{\ran T^*}$ contains the entire space $\overline{\spn}\{ \pi(f) u_j : f \in C_c(G)\}$. By density, it also contains $\overline{\spn}\{ \pi(f) u_j : f \in L^1(G) \}$, which equals $\overline{\spn}\{ \pi(x) u_j : x \in G\}$ (see~\cite[Thm.\ 3.12(c)]{F}). In particular, $\overline{\ran T^*}$ contains $\pi(x) u_j$ for all $j \in I$ and all $x \in G$. Since $E(\A)$ is complete, it follows that $\overline{\ran T^*} = \H$. Therefore, $U$ is an isometry.
\end{proof}

We are almost ready to prove Theorem~\ref{thm: SI embedding}. First we require a measure-theoretic result from \cite{I}. Let $\mathcal{G}$ be as in the hypothesis of Theorem~\ref{thm: SI embedding}. We write $G\backslash \mathcal{G}$ for the topological space of \emph{right} cosets for $G$ in $\mathcal{G}$, and $q\colon \mathcal{G} \to G\backslash \mathcal{G}$ for the quotient map. By \cite{FG}, there is a Borel measurable function $\tau \colon G\backslash \mathcal{G} \to \mathcal{G}$ mapping compact sets in $G\backslash \mathcal{G}$ to pre-compact sets in $\mathcal{G}$, with the property that $\tau \circ q = \text{id}$. Hence $\tau$ gives a measurable choice of right coset representatives for $G$ in $\mathcal{G}$. Then \cite[Theorem 3.4]{I} provides a unique regular Borel measure $\mu_{G\backslash \mathcal{G}}$ on $G\backslash \mathcal{G}$ such that the Weil-like identity
\begin{equation} \label{eq:RWeil}
\int_\mathcal{G} f\, d\mu_\mathcal{G} = \int_{G\backslash \mathcal{G}} \int_G f(x \tau(Gy))\, d\mu_G(x)\, d\mu_{G\backslash \mathcal{G}}(Gy)
\end{equation}
holds for all $f\in L^1(\mathcal{G})$. With this measure in mind, there is a unitary $L^2(G\times G\backslash \mathcal{G}) \cong L^2(\mathcal{G})$ that preserves left translation by $G$; see \cite[Corollary~3.7]{I}.

\begin{proof}[Proof of Theorem~\ref{thm: SI embedding}]

Let $I'$ index an orthonormal basis for $L^2(G\backslash \mathcal{G})$, so that $L^2(G\backslash \mathcal{G}) \cong \ell^2(I')$. We claim that $|I'| = [ \mathcal{G} : G]$ when the right-hand side is finite, and that $I'$ is infinite otherwise. Assuming the claim holds for now, our hypothesis implies that $|I| \leq |I'|$, so there is a linear isometry $\ell^2(I) \to L^2(G \backslash \mathcal{G})$. Consequently, we have a sequence of linear isometries
\[ L^2(G\times I) \cong L^2(G;\ell^2(I)) \to L^2(G; L^2(G \backslash \mathcal{G})) \cong L^2(G\times G \backslash \mathcal{G}) \cong L^2(\mathcal{G}) \]
that all preserve left translation by $G$. Write $U_1 \colon L^2(G\times I) \to L^2(\mathcal{G})$ for the resulting composition, and let $U_2 \colon \H \to L^2(G\times I)$ be the isometry from Lemma~\ref{lem:Bessel regular rep}. Then $U:=U_1U_2 \colon \H \to L^2(\mathcal{G})$ is a linear isometry intertwining $\pi$ with left translation by $G$, and the desired result follows by taking $V = U(\H)$.

It remains to prove our claim.  First, we we will show that every nonempty open subset $U \subset G \backslash \mathcal{G}$ has positive measure. That nonempty open sets have positive measure in $\mathcal{G}$ is well known \cite[Proposition 2.19]{F}. As $\mu_\mathcal{G}$ is regular, it follows that $q^{-1}(U)$ contains a compact subset $K$ with $0 < \mu_\mathcal{G}(K) < \infty$. By \eqref{eq:RWeil},
\[ 0 < \mu_\mathcal{G}(K) = \int_\mathcal{G} \chi_{K}\, d\mu_\mathcal{G} = \int_{G \backslash \mathcal{G}} \int_G \chi_K(x \tau(G y))\, d\mu_G(x)\, d\mu_{G \backslash \mathcal{G}}(G y) \]
\[ \leq \int_{G \backslash \mathcal{G}} \int_G \chi_{q^{-1}(U)}(x \tau(G y))\, d\mu_G(x)\, d\mu_{G \backslash \mathcal{G}}(G y). \]
For any $x \in G$ and $Gy \in G \backslash \mathcal{G}$, we have $x \tau(Gy) \in q^{-1}(U)$ if and only if $q(x \tau(Gy)) \in U$. Since $q(x \tau(Gy)) = q(\tau(Gy)) = Gy$, we have $\chi_{q^{-1}(U)}(x \tau(Gy)) = \chi_U(Gy)$. The inequality above now reads
\[ 0 <  \int_{G \backslash \mathcal{G}} \int_G \chi_U(G y)\, d\mu_G(x)\, d\mu_{G \backslash \mathcal{G}}(G y) = \mu_{G \backslash \mathcal{G}}(U) \cdot \mu_G(G), \]
so $\mu_{G \backslash \mathcal{G}}(U) > 0$, as desired.

If $[\mathcal{G} : G] = n < \infty$, then $G$ is open in $\mathcal{G}$ as the complement of a union of finitely many closed cosets. Applying right translation, we see that all right cosets of $G$ are open in $\mathcal{G}$. Thus, $G \backslash \mathcal{G}$ is discrete with $n$ points. Each of these has positive measure by the above, so $\dim L^2(G\backslash \mathcal{G}) = n$.

Now assume that $[\mathcal{G} : G]$ is infinite. As a topological space, $G\backslash \mathcal{G}$ is homeomorphic to $\mathcal{G}/G$ through the mapping $Gx \mapsto x^{-1}G$. Since $G$ is closed, $\mathcal{G}/G$ is Hausdorff by \cite[Theorem~5.21]{HR1}, so $G\backslash \mathcal{G}$ is, too. Thus, any open set in $G\backslash \mathcal{G}$ with at least two points contains two disjoint open subsets, each of which has positive measure. Therefore, $G \backslash \mathcal{G}$ either contains an infinite sequence of open sets $U_1 \supset U_2 \supset \dotsb$ with the property that $0 < \mu_{G \backslash \mathcal{G}}(U_n) < \infty$ and $\mu_{G \backslash \mathcal{G}}(U_n \setminus U_{n+1}) > 0$ for each $n$, or $G \backslash \mathcal{G}$ contains an open set which is a point. In the first case, the characteristic functions $\chi_{U_1}, \chi_{U_2},\dotsc$ all belong to $L^2(G \backslash \mathcal{G})$ and satisfy $\chi_{U_{n+1}} \notin \spn\{\chi_{U_1},\dotsc,\chi_{U_n}\}$ for all $n$, so $L^2(G \backslash \mathcal{G})$ is infinite dimensional. In the second case, at least one right coset of $G$ is open in $\mathcal{G}$. As above, this implies that every point in $G \backslash \mathcal{G}$ has positive measure, so $\dim L^2(G \backslash \mathcal{G}) = \infty$. This completes the proof.
\end{proof}

\subsection{Abelian groups}

We now specialize to the case where $G$ is abelian, with the following conventions. We write $\hat{G}$ for the Pontryagin dual group of continuous homomorphisms $\alpha \colon G \to \mathbb{T}$. It is a locally compact abelian group under the operation of pointwise multiplication, and the topology of uniform convergence on compact sets. The Fourier transform of $f \in L^1(G)$ is the function $\hat{f} \in C_0(\hat{G})$ given by
\[ \hat{f}(\alpha) = \int_G f(x) \overline{\alpha(x)}\, dx \qquad (\alpha \in \hat{G}). \]
The Fourier transform extends by continuity to a linear mapping $L^2(G) \to L^2(\hat{G})$. We assume that Haar measure on $\hat{G}$ is scaled to make this mapping a unitary. Finally, given a Hilbert space $\K$ we associate each $x \in G$ with the \emph{modulation} $M_x \colon L^2(\hat{G};\K) \to L^2(\hat{G};\K)$ given by
\[ (M_x\varphi)(\alpha) = \overline{\alpha(x)}\cdot \varphi(\alpha) \qquad (\varphi \in L^2(\hat{G};\K), \alpha \in \hat{G}). \]
The mapping $x \mapsto M_x$ is a unitary representation of $G$. Taking $\K=\C$, we have the familiar intertwining property 
\[ (L_x f)\caret = M_x \hat{f} \qquad (x \in G, f \in L^2(G)). \]

When $G$ is abelian, Theorem~\ref{thm: SI embedding} reduces $E(\A)$ to a system of translates by an abelian subgroup. Reproducing systems of this form have been the focus of considerable research in the last two decades~\cite{BHP2,B,BR,CP,I,KR}. The by-now standard approach is to apply a unitary $L^2(\mathcal{G})\cong L^2(\hat{G};\H)$ that converts translation by $G$ into modulation. The existence of such a unitary, for any choice of second-countable $\mathcal{G} \supset G$, was proven in~\cite{I}. In that sense, the theorem below can be seen abstractly as a corollary to Theorem~\ref{thm: SI embedding}. We provide a more direct proof below.

\begin{theorem} \label{thm:LCAMod}
If $G$ is abelian, then for every separable Hilbert space $\mathcal{K}$ with $\dim \mathcal{K} \geq |I|$ there exists a linear isometry $U \colon \H \to L^2(\hat{G}; \mathcal{K})$ that intertwines $\pi$ with modulation. Consequently, $E(\A)$ is unitarily equivalent to the system of modulations $\{ M_x U u_i \}_{x \in G, i \in I}$ in the modulation-invariant subspace $U(\H)$.
\end{theorem}

\begin{proof}
Without loss of generality, we may assume that $\mathcal{K} = \ell^2(I')$ for a countable set $I'$ satisfying $|I'| \geq |I|$.
By Lemma~\ref{lem:Bessel regular rep}, $\pi$ is unitarily equivalent to the translation action of $G$ on a closed invariant subspace of $L^2(G\times I)$. Any choice of injection $I\to I'$ determines a linear isometry $L^2(G\times I) \to L^2(G\times I')$ that preserves left translation by $G$. To the latter space we apply a sequence of unitaries
\[ L^2(G\times I') \overset{U_1}{\cong} \ell^2(I'; L^2(G)) \overset{U_2}{\cong} \ell^2(I'; L^2(\hat{G})) \overset{U_3}{\cong} L^2(\hat{G}; \ell^2(I')), \]
where $U_1$ and $U_3$ are the natural identifications and $U_2$ is the ``entrywise'' Fourier transform $U_2\{ f_i \}_{i\in I'} = \{ \hat{f_i} \}_{i\in I'}$.
The resulting unitary $L^2(G\times I') \cong L^2(\hat{G}; \ell^2(I'))$ intertwines left translation with modulation. The composition
\[ \H \to L^2(G \times I) \to L^2(G\times I') \to L^2(\hat{G}; \ell^2(I')) \]
is the desired isometry $U \colon \H \to L^2(\hat{G}; \ell^2(I'))$.
\end{proof}

A major goal of this paper is to classify Bessel $G$-systems for a given abelian group $G$. By the last result, we can focus our attention on systems of modulations in modulation-invariant subspaces of $L^2(\hat{G};\ell^2(I))$. These subspaces were studied in~\cite{BR}, and the corresponding systems of modulations in~\cite{I}. Any unitary equivalence between modulation systems necessarily commutes with the modulation operators. In order to classify such systems, we therefore require an understanding of morphisms between modulation-invariant spaces. In other words, we must broaden our vantage to see modulation-invariant spaces as a category~\cite{ML}. The next few sections focus on this goal, working in an even broader, measure-theoretic context. We return to classify Bessel $G$-systems in Section~\ref{sec:LCAFrm}.

\section{Morphisms of multiplication-invariant spaces} \label{sec: MI morphisms}

We now turn our attention to a measure-theoretic generalization of modulation.

\begin{assume}
\label{assumptions}
Throughout Sections~3--6, we fix separable Hilbert spaces $\H,\H'$ and a positive, $\sigma$-finite, and complete measure space $(X,\mathcal{M},\mu)$ for which $L^2(X)$ is separable. 
\end{assume}

For example, $X = \mathbb{R}$ or $[0,1]$ with Lebesgue measure. More generally, $X$ can be any second countable locally compact group equipped with the completion of a left Haar measure. The condition that $L^2(X)$ is separable occurs whenever $X$ is the completion of a standard measure space equipped with a sigma-finite measure~\cite[Prop.~3.4.5]{Cohn}.

We are interested in the Hilbert space $L^2(X;\H)$, which provides the setting for a measure-theoretic abstraction of modulation: instead of multiplying $L^2(\hat{G};\ell^2(I))$ by characters of $\hat{G}$, we multiply $L^2(X;\H)$ by elements of $L^\infty(X)$. Specifically, every $\phi \in L^\infty(X)$ gives a \emph{multiplication operator} $M_\phi \in B(L^2(X;\H))$, which is defined by the formula
\[ (M_\phi \varphi)(x) = \phi(x) \varphi(x) \qquad (\varphi \in L^2(X;\H),\, \text{a.e.\ }x \in X). \]
We write $M_\phi'$ for the corresponding operator on $L^2(X;\H')$.

\begin{defn}
A closed subspace $V \subset L^2(X;\H)$ is called \emph{multiplication-invariant} (MI) if $M_\phi V \subset V$ for every $\phi \in L^\infty(X)$. Given two MI spaces $V \subset L^2(X;\H)$ and $V' \subset L^2(X;\H')$, a \emph{multiplication-invariant operator} (\emph{MI operator}) between them is a bounded linear operator $T \colon V \to V'$ such that $TM_\phi = M_\phi' T$ for all $\phi \in L^\infty(X)$.
\end{defn}

MI spaces have been characterized in \cite{He2,BR}. Our purpose here is to give a compatible characterization of MI operators. 

In general, we are interested in preserving multiplication by a large subset of $L^\infty(X)$, akin to the characters when $X = \hat{G}$. The following notion was introduced in \cite{BR} for this purpose.

\begin{defn}
A \emph{determining set} for $L^1(X)$ is a subset $\D$ of the dual $L^\infty(X)$ which separates points in $L^1(X)$: given $f_1 \neq f_2$ in $L^1(X)$, there exists $g \in \D$ such that
\[ \int_X f_1\overline{g}\ d\mu \neq \int_X f_2\overline{g}\ d\mu. \]
\end{defn}

\begin{defn}\label{rf}
An $\H$-valued \emph{range function} on $X$ is a mapping 
\[
J \colon X \to \{ \text{closed subspaces of }\H\}.
\]
Equivalently, it is a choice of orthogonal projection $P_J(x) \in B(\H)$ for every $x \in \H$, where $P_J(x)$ projects onto $J(x)$. We call $J$ \emph{measurable} if $P_J$ is weakly measurable: for any $u,v\in \H$, the mapping $x \mapsto \langle P_J(x) u, v \rangle$ is measurable on $X$.
\end{defn}

With every range function $J\colon X \to \{ \text{closed subspaces of }\H\}$ we associate a closed subspace
\[ V_J := \{ \varphi \in L^2(X;\H) \colon \varphi(x) \in J(x) \text{ for a.e.\ }x\in X\}. \]
A moment's reflection shows that $V_J$ is multiplication invariant. The fundamental characterization of MI spaces says that this construction exhausts all possibilities.

\begin{theorem}[\cite{He2,BR}] \label{thm:MISpa}
Let $V\subset L^2(X;\H)$ be a closed subspace. For every determining set $\D$, the following are equivalent:
\begin{enumerate}[(i)]
\item $V$ is an MI space.
\item For every $\phi \in \D$, $M_\phi V \subset V$.
\item There is a measurable range function $J$ such that $V = V_J$.
\end{enumerate}
More precisely, suppose that $V$ is an MI space which is generated by an (at most) countable collection of functions $\{\varphi_i\}_{i\in I}$, in the sense that $\{M_\phi \varphi_i: \phi \in L^\infty(X), i\in I\}$ is dense in $V$. Then the corresponding range function $J$ satisfies
\begin{equation}\label{MI0}
J(x) = \ov{\spn}\{ \varphi_i(x): i\in \N\} \qquad\text{for a.e.\ }x\in X.
\end{equation}
Moreover, the mapping $J\mapsto V_J$ gives a bijection between the set of measurable $\H$-valued range functions on $X$ (up to equality a.e.) and the set of MI subspaces of $L^2(X;\H)$.
\end{theorem}

We introduce the following notion in order to derive an operator version of Theorem~\ref{thm:MISpa}.

\begin{defn}
Given measurable range functions
\[ J \colon X \to \{ \text{closed subspaces of }\H\} \qquad \text{and} \qquad J' \colon X \to \{ \text{closed subspaces of }\H' \}, \]
a \emph{range operator} $R \colon J \to J'$ is a choice of linear operators $R(x) \colon J(x) \to J'(x)$ for each $x \in X$. We say $R$ is \emph{bounded} if
\[ \Norm{R}:= \esssup_{x\in X} \Norm{R(x)}_{\text{op}} < \infty. \]
It is \emph{measurable} if for every $u\in \H$, $v \in \H'$ the function $x \mapsto \langle R(x) P_{J}(x) u, v \rangle$ is measurable on $X$.
\end{defn}

Just as with range functions and MI spaces, every bounded, measurable range operator $R\colon J \to J'$ defines a bounded MI operator $\int_X^\oplus R(x)\, d\mu(x) \colon V_{J} \to V_{J'}$, given by
\[ \left[ \int_X^\oplus R(x) d\mu(x)\, \varphi\right](y) = R(y)[\varphi(y)] \qquad (\varphi \in V_{J},\ y \in X). \]
We call such operators \emph{decomposable}.

Our main result in this section is the following analogue of~\cite[Theorem~4.5]{B}.

\begin{theorem} \label{thm:MIOp}
Let
\[ J \colon X \to \{ \text{closed subspaces of }\H\} \qquad \text{and} \qquad J' \colon X \to \{ \text{closed subspaces of }\H' \} \]
be measurable range functions, and let $T \colon V_J \to V_{J'}$ be a bounded linear operator. For every determining set $\D$, the following are equivalent:
\begin{enumerate}[(i)]
\item $T$ is an MI operator.
\item For every $\phi \in \D$, $TM_\phi = M_\phi' T$.
\item There is a bounded measurable range operator $R \colon J \to J'$ such that $T = \int_X^\oplus R(x)\, d\mu(x)$.
\end{enumerate}
Moreover, the mapping $R \mapsto \int_X^\oplus R(x)\, d\mu(x)$ gives a one-to-one correspondence between bounded measurable range operators and MI operators, provided we identify range operators that agree a.e.\ on $X$.
\end{theorem}

In order to prove Theorem~\ref{thm:MIOp}, we will consider each $V_J$ to be a direct integral of Hilbert spaces, in the following sense. Fix an orthonormal basis $\{u_j\}_{j\in I}$ for $\H$. Given a measurable $\H$-valued range function $J$, we define $\int_X^{\oplus} J(x) d\mu(x)$ to be the space of all functions $\varphi \in \prod_{x \in X} J(x)$ such that $\int_X \Norm{\varphi(x)}^2 d\mu(x) < \infty$, and such that, for each $j \in I$, the mapping $x \mapsto \langle \varphi(x), P_J(x) u_j \rangle$ is measurable on $X$. As with $L^p$ spaces, we consider functions in $\int_X^\oplus J(x) d\mu(x)$ to be identical when they differ only on a set of measure zero. The direct integral then becomes a Hilbert space with inner product $\langle \varphi, \psi \rangle = \int_X \langle \varphi(x), \psi(x) \rangle d\mu(x)$. For details, consult~\cite{F}. The following was sketched in \cite{BR}.

\begin{lemma} \label{lem:DI}
Given a measurable $\H$-valued range function $J$, we have $V_J = \int_X^{\oplus} J(x) d\mu(x)$. In particular, the latter space does not depend on the choice of orthonormal basis for $\H$.
\end{lemma}

\begin{proof}
Suppose that $\varphi \colon X \to \H$ has the property that $\varphi(x) \in J(x)$ for all $x\in X$. Then for any $j \in I$ and $x \in X$ we have
\[ \langle \varphi(x), P_J(x) u_j \rangle = \langle P_J(x) \varphi(x), u_j \rangle = \langle \varphi(x), u_j \rangle. \]
If $\varphi \in \int_X^\oplus J(x) d\mu(x)$, it follows that for each $v \in \H$, the  mapping $x \mapsto \langle \varphi(x), v \rangle$ is measurable on $X$. Using Pettis's measurability theorem~\cite[Theorem~1.1]{Pe}, we deduce that $\varphi$ is measurable into $\H$. Consequently, $\varphi \in V_J$. The converse is proved similarly.
\end{proof}

\begin{rem}
Let $R \colon J_1 \to J_2$ be a bounded, measurable range operator. From the perspective of Lemma~\ref{lem:DI}, the operator we have called $\int_X^\oplus R(x)\, d\mu(x)$ is quite literally the direct integral of the measurable field of operators given by $R$. In other words, what we have called \emph{range operators} are in one-to-one correspondence with what are usually called \emph{decomposable operators} in the von Neumann algebra literature, provided we identify range operators that agree a.e.\ on $X$. For background on this and other matters of von Neumann algebras, we refer the reader to \cite{D2}.
\end{rem}

\begin{lemma}[ Proposition 2.2 of \cite{BR} ] \label{lem:rangProj}
If $J$ is a measurable range function, then orthogonal projection onto $V_J$ is given by $\int_X^\oplus P_J(x) d\mu(x)$.
\end{lemma}

Finally, we require the following characterization of determining sets, one direction of which was noted in~\cite{BR}.

\begin{lemma} \label{lem:WSD}
A subset $\D \subset L^\infty(X)$ is a determining set for $L^1(X)$ if and only if its finite linear span is weak-$\ast$ dense in $L^\infty(X) \cong L^1(X)^*$.
\end{lemma}

\begin{proof}
Let $N\subset L^\infty(X) = L^1(X)^*$ be the finite linear span of $\D$. Its annihilator in $L^1(X)$ is
\begin{align*}
{}^\perp N &:= \left\{ f \in L^1(X) : \int_X f(x) \overline{\phi(x)}\, d\mu(x) = 0 \text{ for all } \phi \in N \right\} \\
&= \left\{ f \in L^1(X) : \int_X f(x) \overline{\phi(x)}\, d\mu(x) = 0 \text{ for all } \phi \in \D \right\}.
\end{align*}
To say that $\D$ is a determining set means precisely that ${}^\perp N = \{0\}$, or equivalently, that the double annihilator
\[ ( {}^\perp N)^\perp := \left\{ \phi \in L^\infty(X) : \int_X f(x) \overline{\phi(x)}\, d\mu(x) = 0 \text{ for all }f \in {}^\perp N \right\} \]
is all of $L^\infty(X)$. But $({}^\perp N)^\perp$ is the weak-$\ast$ closure of $N$ \cite[Theorem~4.7]{Ru}. Hence, ${}^\perp N=\{0\}$ if and only if $N$ is dense in $L^\infty(X)$ under the weak-$\ast$ topology.
\end{proof}

\begin{proof}[Proof of Theorem~\ref{thm:MIOp}]
By replacing both $\H$ and $\H'$ by $\H\oplus \H'$ if necessary, we may assume without loss of generality that $\H = \H'$.

To begin, extend $T$ to a bounded linear operator $\bar{T}$ on all of $L^2(X;\H)$ by setting $\bar{T} = 0$ on the orthogonal complement of $V_{J}$. For this extended operator, we consider the analogous statements of (i)---(iii), namely:
\begin{enumerate}[(i')]
\item $\bar{T}$ is an MI operator.
\item For every $\phi \in \D$, $\bar{T}M_\phi = M_\phi \bar{T}$.
\item There is a bounded measurable range operator $\bar{R}$ such that $\bar{T} = \int_X^\oplus \bar{R}(x)\, d\mu(x)$.
\end{enumerate}
We will show that (i')---(iii') are equivalent, and then we will show that each of (i)---(iii) is equivalent to its primed version.

(i') $\iff$ (ii') $\iff$ (iii').
Let $M \colon L^\infty(X) \to B(L^2(X;\H))$ be the embedding of $L^\infty(X)$ into $B(L^2(X;\H))$ as a von Neumann algebra of multiplication operators. Then (i') says that $\bar{T}$ belongs to the commutant
\[ M(L^\infty(X))' = \{ S \in B(L^2(X;\H)) : S M_\phi = M_\phi S \text{ for all } \phi \in L^\infty(X)\}, \]
(ii') says that $\bar{T} \in M(\D)'$, and (iii') says that $\bar{T}$ is decomposable. By Lemma~\ref{lem:WSD}, the finite linear span of $\D$ is dense in $L^\infty(X)$ under its weak-$\ast$ topology, which is identical to the weak operator topology on $M(L^\infty(X))$; see Lemma~\ref{lem:WOTWS} in the appendix. Hence, $M(\D)'=M(L^\infty(X))'$, which is well known to be the algebra of decomposable operators~\cite[Thm~II.2.1]{D2}. In other words, (i')---(iii') are equivalent.

(ii') $\implies$ (ii). 
If (ii') is true, then for any $\varphi \in V_1$ and $\phi \in \D$ we have
\[ TM_\phi \varphi = \bar{T} M_\phi \varphi  = M_\phi \bar{T} \varphi = M_\phi T \varphi. \]

(ii) $\implies$ (ii'). 
Conversely, suppose (ii) holds. Given any $\varphi \in L^2(X;\H)$, we can decompose $\varphi = \varphi' + \varphi''$ with $\varphi' \in V_J$ and $\varphi'' \in V_J^\perp$. The orthogonal complement of an MI space is again MI by \cite[Lemma 2.5]{BR}. Thus, for any $\phi \in \D$ we have $M_\phi \varphi' \in V_J$ and $M_\phi \varphi'' \in V_J^\perp$, so that
\[ \bar{T} M_\phi \varphi = \bar{T} M_\phi \varphi' + \bar{T} M_\phi \varphi'' = TM_\phi \varphi'. \]
Applying (ii), we see that
\[ \bar{T} M_\phi \varphi = TM_\phi \varphi' = M_\phi T \varphi' = M_\phi \bar{T} \varphi. \]

(i) $\iff$ (i') follows from the above by taking $\D = L^\infty(X)$.

(iii) $\implies$ (iii').
Suppose (iii) holds for a measurable range operator $R \colon J \to J'$. For each $x\in X$, let $\bar{R}(x) = R(x) P_{J}(x) \in B(\H)$. Then $\bar{R}$ is measurable, and we claim that $\bar{T} = \int_X^{\oplus} \bar{R}(x) d\mu(x)$. To see this, let $P = \int_X^{\oplus} P_{J}(x) d\mu(x)$ be orthogonal projection onto $V_J$, as in Lemma~\ref{lem:rangProj}. Then for any $\varphi \in L^2(X;\H)$ and a.e.\ $x \in X$ we have
\[ (\bar{T} \varphi)(x) = (T P \varphi)(x) = R(x) (P \varphi)(x) = R(x) P_{J}(x) \varphi(x) = \bar{R}(x) \varphi(x). \]

(iii') $\implies$ (iii).
Conversely, assume (iii'). Since the image of $\bar{T}$ is contained in $V_{J'}$, and since orthogonal projection onto $V_{J'}$ is given pointwise by projection onto $J'(x)$, we conclude that $\bar{R}(x)$ takes its image in $J'(x)$ for a.e.\ $x \in X$. Hence we can define $R(x) \colon J(x) \to J'(x)$ to be the restriction of $\bar{R}(x)$ for a.e.\ $x \in X$. The resulting range operator $R \colon J \to J'$ is easily seen to be measurable, and for any $\varphi \in V_J$ and a.e.\ $x \in X$ we obtain
\[ (T \varphi)(x) = (\bar{T} \varphi)(x) = \bar{R}(x) \varphi(x) = R(x) \varphi(x). \]
Thus, $T = \int_X^{\oplus} R(x) d\mu(x)$. This completes the proof that (i)---(iii) are equivalent.

Finally, for the uniqueness claim, suppose we have measurable range operators $R,R' \colon J \to J'$ such that $R(x) \neq R'(x)$ on a set of positive measure. After choosing an orthonormal basis $\{e_n\}_{n=1}^\infty$ for $\H$, a standard argument shows there must be a basis vector $e_n$ with ${[R(x) - R'(x)]e_n} \neq 0$ on a set $E\subset X$ of finite, positive measure. Letting $\varphi \in L^2(X;\H)$ be given by 
\[ \varphi(x) = \begin{cases} e_n, & \text{if }x \in E \\ 0, & \text{otherwise}, \end{cases} \]
we then have $R(x) \varphi(x) \neq R'(x) \varphi(x)$ for all $x \in E$. Consequently, 
\[ \int_X^\oplus R(x) d\mu(x) \neq \int_X^\oplus R'(x) d\mu(x). \qedhere \]
\end{proof}

Theorem~\ref{thm:MISpa} and Theorem~\ref{thm:MIOp} have a convenient formulation in terms of category theory. For background, we refer the reader to~\cite{ML}.

\begin{defn}
We write $X$-\textbf{MI} for the category whose objects are MI subspaces of $L^2(X;\K)$ for various separable Hilbert spaces $\K$, with MI operators as morphisms. 
We write $X$-\textbf{Ran} for the category whose objects are equivalence classes of measurable range functions $J \colon X \to \{ \text{closed subspaces of }\K\}$ for various separable Hilbert spaces $\K$, and whose morphisms are equivalence classes of measurable range operators. Here, range functions (resp.\ operators) are considered equivalent when they agree a.e.\ on $X$.
\end{defn}

\begin{cor} \label{cor:CatIso}
The functor $F\colon X\text{-}\textbf{Ran} \to X\text{-}\textbf{MI}$ given by $F(J) = V_J$ and $F(R) = \int_X^\oplus R(x)\, d\mu(x)$ is an isomorphism of categories.
\end{cor}

\section{Pointwise properties of multiplication-invariant operators} \label{sec:pointwise properties}

We retain Assumption~\ref{assumptions}.
The purpose of this section is to examine the isomorphism of categories from Corollary \ref{cor:CatIso} in greater detail. Theorem \ref{prop} shows that this isomorphism preserves a lot of operator properties, in the sense that global properties of $\int_X^\oplus R(x)\, d\mu(x)$ often reduce to pointwise a.e.\ properties for $R(x)$.

\begin{theorem}\label{prop}
Let $V \subset L^2(X;\H)$ and $V' \subset L^2(X;\H')$ be two MI spaces with range functions $J$ and $J'$, resp.
Let $T:V \to V'$ be an MI operator and $R:J \to J'$ be the corresponding measurable range operator as in Theorem \ref{thm:MIOp}. Then, the following are true.
\begin{enumerate}[(i)]
\item $||T||_{op}= \esssup_{x\in X} ||R(x)||_{op}$.
\item $T$ is bounded from below, i.e., there exists a constant $C>0$ such that
\begin{equation}\label{prop1}
||T \varphi|| \ge C ||\varphi|| \qquad\text{for all }\varphi \in V,
\end{equation}
if and only if for a.e.\ $x\in X$,
\begin{equation}\label{prop2}
||R(x)v || \ge C ||v|| \qquad\text{for all } v\in J(x).
\end{equation}
\item $T$ is invertible if and only if $R(x)$ is invertible for a.e.\ $x\in X$ and 
\begin{equation}\label{prop3}
\esssup_{x\in X} \Norm{R(x)^{-1}}_{op}<\infty.
\end{equation}
In this case, $T^{-1}$ is also an MI operator and $x\mapsto R(x)^{-1}$ is its range operator.
\item the adjoint $T^*: V' \to V$ is an MI operator with the corresponding range operator $R^*: J' \to J$ given by $R^*(x)=(R(x))^*$ for a.e.\ $x\in\ X$.
\item $T$ is unitary if and only if $R(x)$ is unitary for a.e.\ $x\in X$. 
\item $T$ is normal if and only if $R(x)$ is normal for a.e.\ $x\in X$. 
\item $T$ is 1-to-1 if and only if $R(x)$ is 1-to-1 for a.e.\ $x\in X$.
\item $T$ is an isometry if and only if $R(x)$ is an isometry for a.e.\ $x\in X$.
\item $T$ is a partial isometry if and only if $R(x)$ is a partial isometry for a.e.\ $x\in X$.
\end{enumerate}
\end{theorem}

\begin{proof} To prove (i), let $||R||=\esssup_{x\in X} ||R(x)||_{op}$. Then for any $\varphi \in V$ we have
\[
||T \varphi||^2 = \int_X ||R(x)[\varphi(x)]||^2_{\mathcal H'} d\mu(x) \le ||R||^2 \int_X ||\varphi(x)||^2_{\mathcal H} d\mu(x) = ||R||^2 ||\varphi||^2.
\]
Hence, $||T||_{op} \le ||R||$. To prove the converse inequality, take any $f\in L^\infty(X)$,
\[
\int_X |f(x)|^2 ||R(x)[\varphi(x)]||^2_{\mathcal H'} d\mu(x) 
= ||T (f\varphi)||^2  \le ||T||_{op}^2 \int_X |f(x)|^2 ||\varphi(x)||^2_{\mathcal H} d\mu(x).
\]
Since $f$ is arbitrary, we deduce that
\begin{equation}\label{prop4}
||R(x)[\varphi(x)]|| \le ||T||_{op} ||\varphi(x)|| \qquad\text{for a.e.\ }x\in X.
\end{equation}
Now we take a collection of functions $\{\varphi_i\}_{i\in \N}$, which generates $V$. That is, $\{M_\phi \varphi_i: \phi \in L^\infty, i\in \N\}$ is dense in $V$. Define a countable sets $\mathcal A$ of finite linear combinations with rational coefficients 
\begin{equation}\label{prop5}
\mathcal A = \bigg \{\varphi = \sum_{i=1}^\infty c_i \varphi_i:   c_i \in \Q \text{ and $c_i=0$ for all but finitely many $i\in \N$} \bigg\}. 
\end{equation}
By the identity \eqref{MI0} in Theorem \ref{thm:MISpa} the set $\{\varphi(x): \varphi \in \mathcal A\}$ is dense in $J(x)$ for a.e.\ $x\in X$. Applying \eqref{prop4} for all $\varphi \in \mathcal A$ yields $||R|| \le ||T||_{op}$.

To prove (ii), suppose that \eqref{prop1} holds. An analogous argument as used in the proof of \eqref{prop4} yields that for any $\varphi \in V$, 
\[
||R(x)[\varphi(x)]|| \ge C ||\varphi(x)|| \qquad\text{for a.e.\ }x\in X.
\]
Applying the above for all $\varphi \in \mathcal A$ yields \eqref{prop2}. Conversely, if \eqref{prop2} holds, then for any $\varphi \in V$ we have
\[
||T \varphi||^2 = \int_X ||R(x)[\varphi(x)]||^2_{\mathcal H'} d\mu(x) \ge C^2 \int_X ||\varphi(x)||^2_{\mathcal H} d\mu(x) = C^2 ||\varphi||^2.
\]

To prove (iii), suppose that $T: V \to V'$ is an invertible MI operator. Clearly, $T^{-1}: V' \to V$ is also an MI operator. Let $S: J' \to J$ be a bounded measurable range operator such that $T^{-1}=\int_X^{\oplus} S(x) d\mu(x)$ which is given by Theorem \ref{thm:MIOp}. Since $T^{-1} \circ T$ and $T \circ T^{-1}$ are the identity operators on $V$ and $V'$, by Corollary \ref{cor:CatIso}, the corresponding range operators $S(x)\circ R(x)$ and $R(x) \circ S(x)$ are the identity operators on $J(x)$ and $J(x')$, respectively, for a.e.~$x\in X$. Hence, $x \mapsto S(x)=R(x)^{-1}$ is a bounded measurable range operator by (i). 
Conversely, suppose that $R(x)$ is invertible for a.e.\ $x\in X$ and \eqref{prop3} holds. By (ii), $T: V \to V'$ is bounded from below. Hence, it suffices to show that $T$ is onto. Since $T(V)$ is a closed MI subspace of $L^2(X; \mathcal H')$, by Theorem \ref{thm:MISpa} there exists a measurable range function $J''$ with values in closed subspaces of $\mathcal H'$ such that $V_{J''} = T(V)$. Let $\{\varphi_i\}_{i\in I}$ be a sequence of generators of $V$. By \eqref{MI0} we have for a.e.~$x\in X$,
\[
J''(x) = \ov{\spn} \{ T \varphi_i(x): i \in \N\} = 
\ov{\spn} \{ R(x)[\varphi_i(x)]: i \in \N\} = R(x)[J(x)]= J'(x).
\]
Hence, $T(V)=V'$.

To prove (iv), take any $\varphi \in V$ and $\psi \in V'$. Then,
\[
\langle T\varphi,\psi \rangle 
= \int_X \langle R(x)[\varphi(x)], \psi(x) \rangle d\mu(x) = 
\int_X \langle \varphi(x), R^*(x)[\psi(x)] \rangle d\mu(x)
=\langle \varphi, T^*\psi \rangle
\]
This combined with the fact that $R^*:J' \to J$ is a measurable range operator shows that $T^*\psi(x)= R^*(x)[\psi(x)]$ for a.e.\ $x\in X$.

To prove (v) and (vi), observe that by Corollary \ref{cor:CatIso} and part (iv) the range operator of an MI operator $T^*T: V \to V$ and $TT^*: V' \to V'$ is $R^*R: J \to J$ and $RR^*: J' \to J'$, respectively. Hence, (v) follows immediately. 

The property (vii) follows immediately from the following lemma.

\begin{lemma}\label{rk}
 Let $V \subset L^2(X;\H)$ and $V' \subset L^2(X;\H')$ be two MI spaces with range functions $J$ and $J'$, resp. Let $T:V \to V'$ be an MI operator and $R:J \to J'$ be the corresponding measurable range operator as in Theorem \ref{thm:MIOp}. Then the following hold:
\begin{enumerate}[(i)]
\item
The space $V''=\ov{T(V)}  \subset L^2(X;\H') $ is an MI space and its range function $J''$ satisfies 
\[
J''(x)= \ov{R(x)[J(x)]} \qquad\text{for a.e.\ }x\in X.
\]
\item
The space $\ker T  \subset L^2(X;\H)$ is an MI space and its range function $K$
\[
K(x)= \ker R(x) \qquad\text{for a.e.\ }x\in X.
\]
\end{enumerate}
\end{lemma}

\begin{proof}
To prove (i), let $\{\varphi_i\}_{i\in \N}$ be a collection of functions which generates $V$. Clearly, $\ov{T(V)}$ is an MI space, which is generated by $\{T\varphi_i\}_{i\in \N}$. Since $T\varphi_i(x)=R(x)[\varphi_i(x)]$ for a.e.\ $x\in X$, by Theorem \ref{thm:MISpa} we have
\[
J''(x) = \ov{\spn} \{ R(x)[\varphi_i(x)]: i \in \N \} = \ov{R(x)[J(x)]}.
\]

To prove (ii), we consider the adjoint $T^*: V' \to V$. By Theorem \ref{prop}(iv), $T^*$ is an MI operator with range operator $R^*: J' \to J$. By part (i), $\ov{T^*(V')}$ is an MI space with the range function $J''$ given by $J''(x)=\ov{R(x)^*[J'(x)]}$ for a.e.\ $x\in X$. Consequently, by Theorem \ref{thm:MISpa} 
\[
T^*(V')^\perp = L^2(X; \mathcal H) \ominus \ov{T^*(V')}
\]
is also an MI space with a measurable range function $(J'')^\perp$ given by
\[
(J'')^\perp (x)=(R(x)^*[J'(x)])^\perp \qquad\text{for a.e.\ }x\in X.
\]
Therefore, the intersection of two measurable range functions
\[
K(x) = \ker R(x) = J(x) \cap (J'')^\perp (x)
\]
is a measurable range function corresponding to the intersection of the corresponding MI spaces
\[
\ker T = V \cap ({T^*(V')})^\perp.
\]
This completes the proof of Lemma \ref{rk}.
\end{proof}

To finish the proof of Theorem \ref{prop} it remains to show the last two properties. The property (viii) is an immediate consequence of (i) and (ii). To prove (ix), suppose that $T$ is partial isometry. That is, the restriction of $T$ to $V \cap (\ker T)^\perp$ is an isometry. By Lemma \ref{rk}, the range function of an MI space $V \cap (\ker T)^\perp$ is $J''$ given by $J''(x)=J(x) \cap (\ker R(x))^\perp$. Thus, the range operator of $T$ restricted to this subspace is a mapping $x \mapsto R(x)|_{J''(x)}$. By (viii) $R(x)$ is a partial isometry for a.e.\ $x\in X$. Reversing this argument shows the opposite implication. Alternatively, one can use the characterization that $T$ is partial isometry if and only if $(T^*T)^2=T^*T$. This combined with Corollary \ref{cor:CatIso} and part (iv) is equivalent to $(R^*(x)R(x))^2=R^*(x)R(x)$ for a.e.\ $x\in X$. This completes the proof of Theorem \ref{prop}.
\end{proof}

Despite the fact that so many properties of MI operators are reflected by range operators, not all properties are preserved. 

\begin{example} If an MI operator $T$ is onto, then its range operator $R(x)$ is onto for a.e.\ $x\in X$, but the converse implication is not true in general. 
A simple example is a multiplication operator $T: L^2([0,1];\mathcal H) \to L^2([0,1];\mathcal H)$ given by 
\[
T \varphi(x)= x \varphi(x), \qquad \text{for } \varphi\in L^2([0,1];\mathcal H),\  x\in [0,1].
\]
However, it is not difficult to show that the combined properties of $T$ being onto and bounded below is characterized by range operator $R$ satisfying uniform bound from below \eqref{prop2} and $R(x)$ is onto for  a.e.\ $x\in X$. We leave the details to the reader.
\end{example}

In the case when an MI operator $T$ acts on the same space MI space $V$, we have additional properties linking the spectrum of $T$ with the spectra of its range operator.

\begin{theorem}\label{spec}
Let $T:V \to V$ be an MI operator and $R: J \to J$ be the corresponding measurable range operator as in Theorem \ref{thm:MIOp}. Then, the following are true.
\begin{enumerate}[(i)]
\item Let $A\le B$ be two real numbers. Then, $T$ is self-adjoint with spectrum $\sigma(T) \subset [A,B]$ $\iff$ $R(x)$ is self-adjoint with spectrum $\sigma(R(x)) \subset [A,B]$ for a.e.\ $x\in X$,
\item Let $K \subset \C$ be a compact set. Then, $T$ is normal  with spectrum $\sigma(T) \subset K$ $\iff$ $R(x)$ is normal with spectrum $\sigma(R(x)) \subset K$ for a.e.\ $x\in X$.
\end{enumerate}
\end{theorem}

While part (i) can be deduced using similar techniques as used in the proof of Theorem \ref{prop}, the proof of part (ii) requires a machinery of set-valued mappings. Hence, we need to postpone the proof of Theorem \ref{spec} and instead review results about measurability of set-valued maps.

\subsection{Measurability of set-valued maps}
We will employ some rudimentary facts about set-valued measurable maps from the book of Aubin and Frankowska \cite[Chapter 8]{AF} and the monograph of Castaing and Valadier \cite[Chapter III]{CV}. The following definition resembles most closely the classical notion of measurability. Recall that $\mathcal{M}$ denotes the underlying $\sigma$-algebra for $X$.

\begin{defn}\label{mea}
Let $Y$ be a topological space, and suppose $F: X \leadsto Y$ is a set-valued map with closed values. That is, $F(x)$ is a closed subset of $Y$ for each $x\in X$. We say that a set-valued map $F$ is measurable if for each open set $O \subset Y$, we have
\begin{equation}\label{mea0}
F^{-1}(O):= \{ x\in X: F(x) \cap O \not= \emptyset \}\in \mathcal M
\end{equation}
\end{defn}

We have the following characterization of set-valued measurable functions with values in closed subsets of a complete separable metric space, see \cite[Theorem 8.1.4]{AF} and \cite[Theorem III.30]{CV}.
The result below uses our assumption that $X$ is a complete measure space.

\begin{theorem}\label{char}
Let $(Y,d)$ be a complete separable metric space, and suppose that $F: X \leadsto Y$ is a set-valued map with non-empty closed images. Then the following are equivalent:

\begin{enumerate}[(i)]
\item $F$ is measurable,

\item the graph of $F$, given by
\[
\operatorname{Graph}(F)=\{(x,v) \in X \times Y: v \in F(x) \},
\]
belongs to the product $\sigma$-algebra $\mathcal M \otimes \mathcal B$, where $\mathcal B$ is the Borel $\sigma$-algebra on $Y$,

\item $F^{-1}(B) \in \mathcal M$ for any Borel set $B \subset Y$, 

\item for all $y\in Y$, the distance map $X \ni x \mapsto d(y,F(x)) \in [0,\infty)$ is measurable,

\item there exists a sequence of measurable selections
  $f_k:X \to Y$, $k\ge 1$, of $F$ such that
\begin{equation}\label{char2}
F(x) = \overline{\{ f_k(x): k\ge 1 \}} \qquad\text{for all }x\in X.
\end{equation}
\end{enumerate}
\end{theorem}

The last equivalence is especially useful and it is known as Castaign's selection theorem. When a set-valued map $F: X \leadsto Y$ takes values in compact sets, there exists yet another equivalent definition of measurability.
The collection of nonempty compact sets of $Y$,
\[
\mathcal P_c(Y) = \{K \subset Y: K \ne \emptyset \text{ is compact} \}
\]
equipped with the Hausdorff distance is a complete and separable metric space, see \cite[Theorem II.8]{CV}.
Recall that the Hausdorff distance between two nonempty compact sets $K_1,K_2 \subset Y$ is defined as
\begin{equation}\label{hd}
d_H(K_1,K_2)=\max\{ \sup_{v\in K_1} \inf_{w\in K_2} d(v,w), \sup_{v\in K_2} \inf_{w\in K_1} d(v,w) \}
\end{equation}
Castaing and Valadier \cite[Theorem III.2]{CV} have shown the following useful characterization of measurability of compact-valued mappings. 
Note that Theorem~\ref{CV} does not require any of our standing assumptions on $X$.

\begin{theorem}\label{CV} 
Let $Y$ be a separable metric space, and suppose that $F: X \leadsto Y$ is a set-valued map with non-empty compact images.  Then, $F$ is measurable in the sense of Definition \ref{mea} if and only if $F$ is measurable as a function into $\mathcal P_c(Y)$ with the Hausdorff distance $d_H$. That is, for every open $U\subset \mathcal P_c(Y)$, $F^{-1}(U)$ is measurable.
\end{theorem}

Finally, we shall prove that the measurability of range functions coincides with a more general concept of measurability in Definition \ref{mea}.

\begin{theorem}\label{BR}
Let $\H$ be a separable Hilbert space. Suppose that $J$ is a range function with values in closed subspaces of $\H$. Then $J$ is measurable as a set-valued mapping $X \leadsto \H$ if and only if $J$ is measurable in the sense of Definition \ref{rf}. 
\end{theorem}

\begin{proof} By Theorem \ref{char}(v), it suffices to show that $J$ is measurable in the sense of Definition \ref{rf} if and only if there exists a sequence of measurable functions $\varphi_k: X \to \mathcal H$, $k\ge 1$, such that
\begin{equation}\label{BR2}
J(x) = \overline{\{ \varphi_k(x): k\ge 1 \}} \qquad\text{for all }x\in X.
\end{equation}
Suppose first that $J$ is a measurable range function. Hence, the mapping $P_J: X \to B(\mathcal H)$ is weakly measurable, where $P_J(x)$ is an orthogonal projection onto $J(x)$ for all $x\in X$. Let $\{v_k:k \ge 1\}$ be a dense subset of $\H$. Then, functions $\varphi_k(x)=P_J(x)v_k$, $k\ge 1$, satisfy \eqref{BR2}. Conversely, suppose that measurable functions $\varphi_k: X \to \mathcal H$, $k\ge 1$, satisfy \eqref{BR2}. 
For each $x\in X$, we apply the Gram--Schmidt process to vectors $\{\varphi_k(x)\}_{k\ge 1}$ to obtain a collection of orthogonal vectors $\{\psi_k(x)\}_{k\ge 1}$ with norms either $0$ or $1$ and such that 
\[
\spn\{\varphi_k: 1\le k \le n \} = \spn\{\psi_k: 1\le k \le n \} \qquad\text{for all }n\ge 1.
\]
The resulting functions $\psi_k: X \to \H$ are measurable and 
\[
P_J(x)v=\sum_{k\ge 1} \langle v,\psi_k(x)\rangle v \qquad v\in \H.
\]
defines an orthogonal projection onto $J(x)$. Consequently, the resulting projection-valued mapping $P_J$ is weakly measurable. 

\end{proof}

\begin{rem} 
Theorem \ref{BR} does not require our standing assumption on a measure space $(X,\mu)$ made in Section \ref{sec: MI morphisms}. Instead, it holds under a minimal assumption
that $X$ is a measurable space. This is a consequence of the fact that the statements (i), (iv), and (v) in Theorem \ref{char} are all equivalent if $X$ is merely a measurable space, see \cite[Theorem III.9]{CV}. 
\end{rem}

In a close analogy to complex-valued functions we introduce the concept of an essential range for complex set-valued functions.

\begin{defn} 
Suppose that a set-valued map $F: X \leadsto \C$  with closed values is measurable. Define the {\it essential range} of $F$ as
\[
\essran(F):=\{ z\in \C: \mu(F^{-1}(D(z,\ve)))>0 \text{ for all }\ve>0\}.
\]
Here, $D(z,\ve)=\{w\in\C: |w-z|<\ve\}$ denotes an open disk in $\C$.
\end{defn}

The following lemma shows basic properties of the essential range. 

\begin{lemma}\label{er}
Let $F: X \leadsto \C$ be a measurable set-valued map with closed images. Then,
\begin{equation}\label{er0}
F(x) \subset \essran(F) \qquad\text{for $\mu$-a.e.\ }x\in X.
\end{equation}
Furthermore, $\essran(F)$ is the smallest closed subset of $\C$ which contains $F(x)$ for a.e.\ $x\in X$. That is,
\begin{equation}\label{er1}
\essran(F) = \bigcap \bigg\{ \ov{ \bigcup_{x\in S} F(x)}: S \subset X \text{ is measurable with } \mu(X \setminus S)=0 \bigg\}.
\end{equation}
\end{lemma}

\begin{proof}
We first prove~\eqref{er0}.
For each $z \notin \essran(F)$, we can choose $\varepsilon_z > 0$ satisfying $\mu(F^{-1}(D(z,\varepsilon_z))) = 0$.
Then $D(z,\varepsilon_z) \subset \mathbb{C} \setminus \essran(F)$.
Indeed, for any $z' \in D(z,\varepsilon_z)$ there exists $\varepsilon' > 0$ with $D(z',\varepsilon') \subset D(z,\varepsilon_z)$, so that
$F^{-1}(D(z',\varepsilon')) \subset F^{-1}(D(z,\varepsilon_z))$ and $\mu(F^{-1}(D(z',\varepsilon'))) = 0$.
In other words, $z' \notin \essran(F)$.
Consequently,
\[ \mathbb{C} \setminus \essran(F) = \bigcup \bigl\{ D(z,\varepsilon_z) : z \notin \essran(F) \bigr\}. \]
Since $\mathbb{C}$ is second countable, we can reduce this to a countable union $\mathbb{C} \setminus \essran(F) = \bigcup_{n=1}^\infty D(z_n, \varepsilon_{z_n})$.
Then
\[ \{ x \in X : F(x) \nsubseteq \essran(F) \} =  F^{-1}(\mathbb{C} \setminus \essran(F)) = \bigcup_{n=1}^\infty F^{-1}(D(z_n,\varepsilon_{z_n})) \] 
is a set of measure zero.
This proves~\eqref{er0}.

A point $z$ does not belong to the set on the right-hand side of \eqref{er1} if and only if there exists a measurable set $S\subset X$ with $\mu(X \setminus S)=0$ such that $z\not\in \ov{ \bigcup_{x\in S} F(x)}$. This is equivalent with the existence of $\ve>0$ such that
\begin{equation}\label{er2}
F(x) \cap D(z,\ve)= \emptyset \qquad\text{for all }x\in S.
\end{equation}
Since $\mu(X \setminus S) =0$, this implies that $\mu(F^{-1}(D(z,\ve))=0$, and hence 
$z\not \in \essran(F)$.

Conversely, suppose that $z\not \in \essran(F)$. Then, $\mu(F^{-1}(D(z,\ve)))=0$ for some $\ve>0$. Let $
S= X \setminus F^{-1}(D(z,\ve))$. Then $\mu(X\setminus S)=0$ and \eqref{er2} holds.
Thus,  $z$ does not belong to the set on the right-had side of \eqref{er1}.
\end{proof}

\subsection{Measurability of the spectrum of range operators}
We are ready to show a result about the spectrum of range operators. 
The results of Azoff \cite[Theorem 3.5]{Az} and Chow \cite[Lemma 2.1]{Chow}, both of which were originally formulated in the language of decomposable operators, are closely related to the following lemma.

\begin{lemma}\label{chow}
Suppose $T:V \to V$ is an MI operator and $R: J \to J$ be the corresponding measurable range operator as in Theorem \ref{thm:MIOp}. Consider the set-valued map $F: X \leadsto \C$ given by $F(x) = \sigma(R(x))$ for $x\in X$. Then, $F$ is measurable and 
\begin{equation}\label{chow1}
\essran(F) \subset \sigma(T).
\end{equation}
In addition, if  the operator-valued resolvent function $x\mapsto (\lambda \mathbf I(x) - R(x))^{-1}$ is $\mu$-essentially bounded for every $\lambda \not \in \essran(F)$, then \eqref{chow1}  is an equality. Here, $\mathbf I(x)$ is the identity on $J(x)$. 
\end{lemma}

\begin{proof}
First, we must show that $F: X \leadsto \C$ is measurable.
By Theorem \ref{char}(ii), it suffices to show that the graph of $F$
\[
\operatorname{Graph}(F)=\{(x,\lambda)\in X \times \C: \lambda \in F(x)=\sigma(R(x)) \} 
\]
belongs to the product $\sigma$-algebra of measurable sets in $X$ and Borel sets in $\C$. Let $\mathcal D$ be a dense countable subset of $\mathcal H$. Let $P_{J}(x)$ be the orthogonal projection of $\mathcal H$ onto $J(x)$. For any $n\in\N$ and $v\in \mathcal D$, consider the set
\[
G_{n,v}=\{(x,\lambda) \in X \times \C: ||(\lambda \mathbf I(x)-R(x))(P_{J(x)}v)|| \ge
 1/n ||P_{J(x)} v|| \},
 \]
where $\mathbf I(x)$ is the identity on $J(x)$.
 Likewise, define
 \[
H_{n,v}=\{(x,\lambda) \in X \times \C: ||(\overline{\lambda} \mathbf I(x)-R^*(x))(P_{J(x)}v)|| \ge
 1/n ||P_{J(x)} v|| \}.
 \]
The sets $G_{n,v}$ and $H_{n,v}$ belong to the product $\sigma$-algebra  on $X\times \C$. 

We will employ the fact that $S:\mathcal H \to \mathcal H$ is invertible if and only if $S$ and $S^*$ are bounded from below. Hence, $\lambda \not\in \sigma(R(x))$ if and only if both $\lambda \mathbf I(x) - R(x)$ and $\ov{\lambda} \mathbf I(x) - R^*(x)$ are bounded from below. Consequently,
\[
(X \times \C) \setminus  \operatorname{Graph}(F) = \bigcup_{n\in \N} \bigcap_{ v\in\mathcal D} (G_{n,v} \cap H_{n,v}).
\]
Hence, $\operatorname{Graph}(F)$ belongs to the product $\sigma$-algebra on $X\times \C$, which shows the measurability of $F$.

The remaining statements such as \eqref{chow1} might possibly be deduced from \cite[Lemma 2.1]{Chow} using Lemma \ref{er}. However, this can be easily shown by the following direct argument.
Suppose that $\lambda_0 \in \C \setminus \sigma(T)$. Choose $\ve>0$ such that $\overline{D(\lambda_0,\ve)} \cap \sigma(T) = \emptyset$. 
Let 
\[
C= \sup_{\lambda \in \overline{D(\lambda_0,\ve)}}||(\lambda \mathbf I - T)^{-1}||<\infty.
\]
By Theorem \ref{prop}(iii) for any $\lambda \in \overline{D(\lambda_0,\ve)}$, we have
\begin{equation}\label{chow3}
||(\lambda \mathbf I(x) - R(x))^{-1}|| \le C<\infty
\qquad\text{for a.e.\ }x\in X.
\end{equation}
Hence,
\[
\overline{D(\lambda_0,\ve)} \cap \sigma(R(x)) = \emptyset 
\qquad\text{for a.e.\ }x\in X.
\]
Thus, $\lambda_0 \not\in \essran(F)$, which shows \eqref{chow1}. 

Finally, suppose that for every $\lambda \not \in \essran(F)$ there exists $C>0$ such that \eqref{chow3} holds. Take any $\lambda \not \in \essran(F)$. By Theorem \ref{prop}(iii), $\lambda I - T$ is invertible. Hence, $\lambda \not \in \sigma(T)$. This proves the equality in \eqref{chow1}.
\end{proof}

Combining Lemmas \ref{er} and \ref{chow} yields the following result which is well-known for decomposable operators, see \cite[Proposition 1]{Gil73} and \cite[Proposition 1.1]{Le}.

\begin{cor}\label{len}
Suppose $T:V \to V$ is an MI operator and $R$ is its corresponding range operator satisfying $T= \int_X^{\oplus} R(x) d\mu(x)$. Then,
\[
\sigma(R(x)) \subset \sigma(T) \qquad\text{for $\mu$-a.e.\ }x\in X.
\]
\end{cor}

The following simple example shows that the condition on essential boundedness of resolvent is necessary in Lemma \ref{chow}.

\begin{example}
Let $X=\N$ be equipped with the counting measure and $\mathcal H=\ell^2(\N)$. Define a range function 
\[
J(n)= \{ (v_k)_{k\in \N} \in \mathcal H: v_k=0 \text{ for all }k \ge n+1\}.
\]
Define a range operator $R(n): J(n) \to J(n)$ by its matrix representation
\[
R(n) = \begin{bmatrix} 1 & 1 & & \\
& 1 & 1 & \\
& & \ddots & \ddots  \\
& & & 1
\end{bmatrix}.
\]
That is, $R(n)$ is a Jordan block of size $n$ for eigenvalue $1$, and hence $\sigma(R(n))=\{1\}$. Let $v\in J(n)$ be a vector 
\[
v_i= \begin{cases} (-1)^i/\sqrt{n} & \text{for } 1\le i \le n,
\\
0 & \text{otherwise.}
\end{cases}
\]
Then $||v||=1$ and $||R(n)v||=1/\sqrt{n}$. Hence the inverse operators $R^{-1}(n)$ are not uniformly bounded. This shows that there exist a sequence $\varphi_n \in L^2(\N; \ell^2)$  such that $||\varphi_n||=1$ and $||T \varphi_n || \to 0$ as $n\to \infty$. Thus, $0 \in \sigma(T)$, and we do not have equality in \eqref{chow1}.
\end{example}

If $T$ is a normal operator, i.e., $T T^*=T^* T$, then we have the equality in \eqref{chow1}.

\begin{cor}\label{nor}
Suppose $T:V \to V$ is a normal MI operator of the form 
\[
T= \int_X^{\oplus} R(x) d\mu(x).
\]
Consider the set-valued map $F: X \leadsto \C$ given by $F(x) = \sigma(R(x))$. Then,
\begin{equation}\label{nor1}
\sigma(T) =\essran(F).
\end{equation}
\end{cor}

\begin{proof}
By the spectral theorem for a normal operator $T$ we have that
\begin{equation}\label{sn}
|| (\lambda \mathbf I - T )^{-1} || = \frac{1}{\operatorname{dist}(\lambda, \sigma(T))}
\qquad\text{for }\lambda \not\in \sigma(T).
\end{equation}
Let $J$ be the range function of a MI space $V$.
By Theorem \ref{prop}(vi), $R(x)$ is normal for a.e. $x\in X$, and hence we have
\begin{equation}\label{sn2}
|| (\lambda \mathbf I(x) - R(x) )^{-1} || = \frac{1}{\operatorname{dist}(\lambda, F(x))}
\qquad\text{for }\lambda \not\in F(x),
\end{equation}
where $\mathbf I(x)$ is the identity on $J(x)$. For any $\lambda \not \in \operatorname{ess\ ran}(F)$, 
there exists $\epsilon>0$ such that $\mu(F^{-1}(D(\lambda,\epsilon)))=0$. By \eqref{sn2}
\[
|| (\lambda \mathbf I(x) - R(x) )^{-1} || \le \frac{1}{\epsilon}
\qquad\text{for a.e. }x\in X.
\]
Therefore, Lemma \ref{chow} yields \eqref{nor1}.
\end{proof}

We are now ready to give the proof of Theorem \ref{spec}.

\begin{proof}[Proof of Theorem \ref{spec}]
It suffices to show only part (ii) as it immediately implies (i) by taking $K=[A,B]$. 
Suppose that $T$ is a normal operator and $\sigma(T) \subset K$ for some compact $K\subset \C$. Then, by Theorem \ref{prop}(vi), $R(x)$ is normal for a.e.\ $x\in X$. Moreover, by \eqref{er0}
 and Corollary \ref{nor}
\[
\sigma(R(x)) \subset \sigma(T) \subset K \qquad\text{for a.e.\ }x\in X.
\]
Conversely, if $R(x)$ is normal for a.e.\ $x$, then by Theorem \ref{prop}(vi) the operator $T$ must be normal as well. If $\sigma(R(x)) \subset K$ for some compact $K\subset \C$ and a.e.\ $x$, then~\eqref{er1} implies $\essran(x\mapsto \sigma(R(x))) \subset K$. By Corollary \ref{nor} we have $\sigma(T) \subset K$.
\end{proof}

\subsection{Dimension function of MI spaces}
Next we show the generalization of \cite[Theorems 4.9 and 4.10]{B} to the setting of MI spaces and MI operators.

\begin{defn}
Let $V \subset L^2(X,\mathcal H)$ be an MI space with the corresponding range function $J$. The {\it dimension function} of $V$ is 
\[
\dim_V: X \to \N \cup \{0,\infty\}, \qquad \dim_V(x) = \dim J(x), \ x\in X.
\]
\end{defn}

\begin{theorem} \label{dim}
Suppose that $V \subset L^2(X;\mathcal H)$ is an MI space and $T: V \to L^2(X;\mathcal H)$ is an MI operator. Then, $V'=\ov{T(V)}$ is an MI space and
\begin{equation}\label{dim0}
\dim_{V'}(x) \le \dim_V(x) \qquad\text{for a.e.\ }x\in X.
\end{equation}
\end{theorem}

\begin{proof}
By Lemma~\ref{rk}(i) the range function of $V'$ satisfies
\[
J'(x) = \ov{R(x)(J(x))} \qquad\text{for a.e.\ }x\in X.
\]
This implies $\dim J'(x) \le \dim J(x)$, which yields \eqref{dim0}.

\end{proof}

\begin{theorem} \label{dimo}
Suppose that $V, V' \subset L^2(X;\mathcal H)$ are MI spaces. The following are equivalent:
\begin{enumerate}[(i)]
\item $\dim_V(x)=\dim_{V'}(x)$ for a.e.\ $x\in X$,
\item there exists an MI isometric isomorphism $T: V \to V'$,
\item there exists an MI isomorphism $T:V \to V'$.
\end{enumerate}
\end{theorem}

\begin{proof}
The equivalence of (i) and (ii) is due to Helson \cite[Theorem 1 in \S1.4]{He2}. (ii) trivially implies (iii). Finally, (iii) implies (i) by applying Theorem \ref{dim} for $T$ and $T^{-1}$.
\end{proof}

\subsection{Functional calculus for MI operators}

We conclude this section by observations about functional calculus for MI operators. Our presentation extends and is influenced by the results in \cite[Section 6]{B3}. We start by recalling two basic forms of functional calculus for operators on Banach and Hilbert spaces, respectively.

\begin{defn} Suppose that $T: B \to B$ is a bounded operator acting on a Banach space $B$. Let $\sigma(T)$ be the spectrum of $T$. Then, for any holomorphic function $h$ defined on some neighborhood $\Omega$ of $\sigma(T)$, define
\begin{equation}\label{ft1}
h(T) = \frac{1}{2\pi i } \int_{\gamma} h(\lambda) (\lambda \mathbf I - T)^{-1} d\lambda,
\end{equation}
where $\gamma$ is any positively oriented contour that surrounds $\sigma(T)$ in $\Omega$. It is known that this definition does not depend on the choice of $\gamma$, see \cite[Section 10.26]{Ru}.
\end{defn}

\begin{defn}
Suppose that $T$ is a normal operator acting on a Hilbert space $\mathcal H$. Let $E$ be the spectral decomposition of $T$, see \cite[Sec. 12.23]{Ru}. Then, for any bounded Borel function $h$ on $\sigma(T)$, define
\begin{equation}\label{ft2}
h(T) = \int_{\sigma(T)} h(\lambda) dE(\lambda).
\end{equation}
\end{defn}

These notations agree wherever they overlap, see~\cite[Theorem 2.6]{Co}.

\begin{theorem}\label{ft} Suppose that $V \subset L^2(X;\mathcal H)$ is an MI space and $T: V \to V$ is an MI operator. Let $R$ be its corresponding range operator so that $T= \int_X^{\oplus} R(x) d\mu(x)$.
Assume that either:
\begin{enumerate}[(i)]
\item $h$ is a holomorphic function on some neighborhood of $\sigma(T)$, or
\item $h$ is a bounded complex Borel function on $\sigma(T)$ and $T$ is normal. 
\end{enumerate}
Then, $h(T)$ is also an MI operator and its corresponding range operator is $x \mapsto h(R(x))$.
\end{theorem}

\begin{proof}
By Corollary \ref{len} and Theorem \ref{spec}, $h(R(x))$ is a well-defined operator for a.e.\ $x$ in both cases (i) and (ii). We need to show $x \mapsto h(R(x))$ is a measurable range operator and 
\begin{equation}\label{ft4}
h(T)= h\left( \int_X^\oplus R(x)\, d\mu(x) \right) = \int_X^\oplus h(R(x))\, d\mu(x). 
\end{equation}
If $\lambda \not \in \sigma(T)$, then
\[
(\lambda \mathbf I -T)^{-1}
= \int_X^\oplus (\lambda \mathbf I(x) - R(x))^{-1} d\mu(x). 
\]
The integral in \eqref{ft1} is approximated by Riemann sums in the operator norm. Outside a set of measure zero $X_0 \subset X$, the mappings $\lambda \mapsto (\lambda \mathbf I(x) - R(x))^{-1}$, $x\in X \setminus X_0$, are continuous at every point $\lambda \not \in \sigma(T)$. Hence, the integrals
\[
h(R(x)) = \frac{1}{2\pi i } \int_{\gamma} h(\lambda) (\lambda \mathbf I(x) - R(x))^{-1} d\lambda
\]
are approximated uniformly a.e.\ by Riemann sums. This yields
 the required conclusion in case (i). 
 
To prove case (ii), observe that 
\[
p(T,T^*) = \int_{\sigma(T)} p(\lambda, \bar \lambda) d E(\lambda),
\]
where $p$ is any polynomial in two variables with complex coefficients and $E$ is the spectral measure of $T$. Clearly, the range operator of $p(T,T^*)$ is $x \mapsto p(R(x),R(x)^*)$.
By the Stone-Weierstrass theorem, functions $\lambda \mapsto p(\lambda,\bar \lambda)$ are dense in $C(\sigma(T))$. Hence, by \cite[Theorem 12.24]{Ru}, the conclusion \eqref{ft4} holds for continuous functions $h$. Finally, it suffices to use two basic facts about functional calculus for Borel functions. First, if $\{h_i\}$ is a uniformly bounded sequence of Borel functions converging pointwise to $h$ on $\sigma(T)$, then $\{h_i(T)\}$ converges to $h(T)$ in the strong operator topology. Second, the space of bounded Borel functions on a compact set $K \subset \C$ is the smallest space $X$ containing $C(K)$ and closed under pointwise limits of uniformly bounded sequences in $X$. In particular, \eqref{ft4} holds if $h=\chi_U$, where $U \subset \C$ is open. By Dunkin's $\pi$-$\lambda$ lemma \eqref{ft4} holds for $h=\chi_U$, where $U \subset \C$ is Borel. Consequently, the required conclusion holds for all bounded Borel functions $h$.
\end{proof}

As an immediate corollary of Theorem \ref{ft} we have a result which is very close to known results for decomposable operators, see \cite[Lemma 2.5]{Chow} and \cite[Proposition 1.4]{Le}.

\begin{cor} Suppose $T$ is a normal MI operator  and $R$ is its corresponding range operator. Let $E$ be the spectral measure of $T$, and $E_{x}$ be the spectral measure of $R(x)$. Then for any Borel set $U \subset \C$, $E(U)$ is an MI operator and its corresponding range operator is $E_x(U)$.
\end{cor}

\begin{proof} It suffices to take $h=\chi_U$ in Theorem \ref{ft}, where $U \subset \C$ is Borel. Indeed, $\chi_U(R(x))$ is an orthogonal projection on $J(x)$ and $E_x$ given by $E_x(U)=\chi_U(R(x))$ is the spectral measure of $R(x)$. By Theorem \ref{ft}, we have $E(U) = \int_X^{\oplus} E_x(U) d\mu(x).$
\end{proof}

\section{Bessel systems of multiplications}

We continue to hold Assumption~\ref{assumptions}.
In this section, we study Bessel systems generated by multiplications in $L^2(X;\H)$. Given a sequence of generators $\A = \{ \varphi_i \}_{i\in I}$ in $L^2(X;\H)$ and a suitable set of functions $\D = \{g_t \}_{t \in Y}$ in $L^\infty(X)$, we study systems of the form $\{ M_{g_t} \varphi_i \}_{t \in Y, i \in I}$. For example, this is the generic form of a Bessel system produced by the unitary action of a locally compact abelian group $G$, by Theorem~\ref{thm:LCAMod}, with $X = \hat{G}$ and $\D$ equal to the characters of $\hat{G}$. When $\D$ is well behaved, the ``global'' system $\{ M_{g_t} \varphi_i \}_{t \in Y, i \in I}$ is Bessel if and only if the ``pointwise'' systems $\{ \varphi_i(x) \}_{i \in I}$ are Bessel for a.e.\ $x \in X$, with uniform bound. As we will see, the properties of these two types of systems are closely related.

The following measure-theoretic generalization of characters on an abelian group was introduced in~\cite{I}.

\begin{defn}
A \emph{Parseval determining set} for $L^1(X)$ consists of another measure space $(Y,\nu)$ and a family $\{g_t\}_{t\in Y}$ in $L^\infty(X)$ such that:
\begin{enumerate}[(i)]
\item For every $f \in L^1(X)$, the mapping $t \mapsto \int_X f \overline{g_t}\, d\mu$ is measurable on $Y$; and
\item For every $f \in L^1(X)$,
\begin{equation} \label{eq:ParsDet}
\int_Y \left| \int_X f(x) \overline{g_t(x)}\, d\mu(x) \right|^2 d\nu(t) = \int_X | f(x) |^2 d\mu(x).
\end{equation}
(Possibly both sides are infinite.)
\end{enumerate}
\end{defn}

\begin{example}
\begin{enumerate}[(a)]
\item
\cite[Lemma~5.2]{I}
When $X=G$ is a locally compact abelian group, its characters form a Parseval determining set for $L^1(G)$ when the dual $Y=\hat{G}$ is equipped with Plancherel measure.
In this case, \eqref{eq:ParsDet} amounts to Plancherel's Theorem.
\item

Let $X$ and $Y$ be countable sets endowed with counting measures. Then, a system $\{g_t\}_{t \in Y}$ in $\ell^2(X) \subset \ell^\infty(X)$ is a Parseval determining set for $\ell^1(X)$ if and only if it is a Parseval frame for $\ell^2(X)$.
\item
Let $G$ be a second countable locally compact group, and let $K\subset G$ be a compact subgroup for which $(G,K)$ is a Gelfand pair~\cite{W07}. Take $X = K\backslash G / K$ to be the space of double cosets of $K$ in $G$, with the quotient topology and Haar measure. Then the space $Y=P(G,K)$ of positive definite spherical functions, equipped with Plancherel measure, forms a Parseval determining set for $L^1(K\backslash G / K)$. When the right-hand side of~\eqref{eq:ParsDet} is finite, this is just Plancherel's Theorem for the spherical transform~\cite[Theorem~9.5.1]{W07}. When the left-hand side is finite, one can emulate the proof of~\cite[Theorem~31.31]{HR2}. We leave details to the reader.
\end{enumerate}
\end{example}

The following is the measure-theoretic generalization of Definition~\ref{defn:GSys}.
\begin{defn} \label{def:CF}
Given a $\sigma$-finite measure space $(\M,m)$ and a separable Hilbert space $\mathcal{K}$, a system $\{u_s\}_{s\in M}$ is called \emph{Bessel} with bound $B > 0$ (briefly, $B$-Bessel) when both of the following hold:
\begin{enumerate}[(i)]
\item For every $v\in \mathcal{K}$, the mapping $s \mapsto \langle v, u_s \rangle$ is measurable on $\M$; and
\item For every $v\in \mathcal{K}$,
\[ \int_\M | \langle v, u_s \rangle |^2\, dm(s) \leq B \Norm{ v}^2. \]
\end{enumerate}
It is a \emph{continuous frame} with bounds $A,B > 0$ (briefly, an $A,B$-frame) if, in addition,
\begin{equation} \label{eq:FD}
A \Norm{v}^2 \leq \int_\M | \langle v, u_s \rangle |^2\, dm(s) \leq B \Norm{v}^2
\end{equation}
for every $v\in \mathcal{K}$. When $A=B$, the frame is called \emph{tight}; when $A=B=1$, it is \emph{Parseval}.
\end{defn}

These notions were introduced independently by Kaiser~\cite{K} and by Ali, Antoine, and Gazeau~\cite{AAG}. The word ``continuous'' is best understood in relation to ``discrete''. The usual notions of (discrete) Bessel systems and frames can be obtained as special cases by taking $m$ to be counting measure. We will use the terms ``frame'' and ``continuous frame'' synonymously.

Continuous frames and Bessel systems behave similarly to their discrete counterparts; see \cite{AAG,GH,K,RND}. In particular, every Bessel system $\{u_s\}_{s\in \M}$ in $\mathcal{K}$ comes with an \emph{analysis operator} $T\colon \mathcal{K} \to L^2(\M)$ given by $(Tv)(s) = \langle v, u_s \rangle$ for $v \in \mathcal{K}$ and $s \in \M$. Its adjoint is the \emph{synthesis operator} $T^* \colon L^2(\M) \to \mathcal{K}$, defined weakly by the vector-valued integral
\[ T^* f = \int_\M f(s) u_s\, dm(s) \qquad (f\in L^2(\M)). \]
Composing these in one direction gives the \emph{frame operator} $S = T^* T \colon \mathcal{K} \to \mathcal{K}$,
\[ Sv = \int_\M \langle v, u_s \rangle u_s\, dm(s) \qquad (v\in \H). \]
Composing in the other direction gives the \emph{Gramian} $T T^* \colon L^2(\M) \to L^2(\M)$. When $\{ u_s \}_{s\in \M}$ is a frame, $S$ is automatically positive and invertible; moreover, $\{ S^{-1/2} u_s \}_{s\in \M}$ is a Parseval frame, called the \emph{canonical tight frame} for $\{ u_s \}_{s\in \M}$.

\begin{assume}
\label{assumptions 2}
For the remainder of Section~5 and all of Section~6 we fix a $\sigma$-finite measure space $(Y,\nu)$ and a Parseval determining set $\D = \{g_t\}_{t\in Y}$ for $L^1(X)$. 
\end{assume}

Given a countable sequence $\A = \{ \varphi_i \}_{i\in I}$ in $L^2(X;\H)$, we consider
\[ E(\A) = \{ M_{g_t} \varphi_i \}_{t\in Y,\, i \in I}, \]
the system of all multiplications of $\A$ by $\D$. We write
\[ S(\A) = \overline{\spn}\{ M_{g_t} \varphi_i \colon t \in Y,\, i \in I\} \]
for the MI space generated by $\A$. According to Theorem \ref{thm:MISpa} and Lemma \ref{lem:DI}, the associated range function for which
\[ S(\A) = \int_X^\oplus J(x)\, d\mu(x) \]
is given by
\begin{equation} \label{eq:J}
J(x) = \overline{\spn}\{ \varphi_i(x) \colon i \in I\} \qquad (\text{a.e.\ }x\in X).
\end{equation}
In particular, every $\varphi_i$ belongs to $S(\A)$.

\begin{theorem} \label{thm:MIF}
Let $\A = \{\varphi_i\}_{i\in I}$ be a countable sequence in $L^2(X;\H)$. For constants $A,B >0$, the following are equivalent:
\begin{enumerate}[(i)]
\item $E(\A)$ is an $A,B$-frame for $S(\A)$ (resp.\ $B$-Bessel).
\item For a.e.\ $x\in X$, $\{\varphi_i(x)\}_{i \in I}$ is an $A,B$-frame for $J(x)$ (resp.\ $B$-Bessel).
\end{enumerate}
\end{theorem}

\begin{proof}
This is Theorem~2.10 of \cite{I}. A careful reading of the proof shows that the upper frame bound can be handled separately.
\end{proof}

\subsection{The Plancherel transform of a Parseval determining set}

When $E(\A)$ is Bessel, Theorem~\ref{thm:MIF} gives us two different kinds of analysis operators: the global analysis operator $T\colon S(\A) \to L^2(Y\times I)$ of $E(\A)$, and the field of pointwise analysis operators $\tilde{T}(x) \colon J(x) \to \ell^2(I)$ associated with $\{\varphi_i(x)\}_{i\in I}$ for a.e.\ $x\in X$. Taken together, the latter describe a decomposable operator $\tilde{T} \colon S(\A) \to {L^2(X;\ell^2(I))}$, by Lemma~\ref{lem:PtAnDec} below. 
Morally, one feels that this should somehow be the ``same'' as the global analysis operator $T$. However, $T$ and $\tilde{T}$ map into completely different spaces, so they cannot literally be equal. 

To resolve this discrepancy, we introduce a variant of the Fourier transform. Given $f\in L^1(X) \cap L^2(X)$, let $\F f \in L^2(Y)$ be the function
\begin{equation} \label{eq:PF}
(\mathcal{F} f)(t) = \int_X f(x) \overline{g_t(x)}\, d\mu(x) \qquad (\text{a.e.\ }t\in Y).
\end{equation}
According to \eqref{eq:ParsDet}, $\Norm{\F f}_{L^2(Y)} = \Norm{ f }_{L^2(X)}$.

\begin{defn}
The \emph{Plancherel transform} associated with $\D$ is the unique linear isometry $\F \colon L^2(X) \to L^2(Y)$ extending~\eqref{eq:PF}.
\end{defn}

When $\D$ is the set of characters on a locally compact abelian group $G$, $\F$ reduces to the Fourier transform $L^2(G) \to L^2(\hat{G})$, which is not only isometric, but unitary. On the other hand, when $X$ and $Y$ are countable indexing sets with their counting measures, $\F$ becomes the analysis operator for the Parseval frame $\D$. Therefore, we cannot generally expect $\F$ to be surjective.

In the general case, $\F$ provides a way to map ${L^2(X;\ell^2(I))}$ into ${L^2(Y\times I)}$. Namely, we can identify $L^2(X;\ell^2(I)) \cong \ell^2(I; L^2(X))$, apply the direct sum 
\[ \bigoplus_{i\in I} \F \colon \ell^2(I; L^2(X)) \to \ell^2(I; L^2(Y)), \]
and then identify $\ell^2(I; L^2(Y)) \cong L^2(Y\times I)$. We will write 
\[ \F_I \colon L^2(X;\ell^2(I)) \to L^2(Y\times I) \]
for this composition, which is again a linear isometry. When $\varphi \in L^2(X;\ell^2(I))$ and the mapping $x \mapsto [\varphi(x)]_i$ belongs to $L^1(X) \cap L^2(X)$ for every $i\in I$, we have
\begin{equation} \label{eq:PFT}
( \F_I \varphi)(t,i) = \int_X [ \varphi(x) ]_i\cdot  \overline{g_t(x)}\, d\mu(x) \qquad (\text{a.e.\ }t \in Y,\, i \in I).
\end{equation}

\subsection{Analysis and synthesis operators}

\begin{theorem} \label{thm:MIAn}
Let $\A = \{ \varphi_i \}_{i\in I}$ be a sequence in $L^2(X;\H)$ for which $E(\A)$ is Bessel, and let $J$ be the range function \eqref{eq:J}.
Then the analysis operator $T\colon S(\A) \to L^2(Y\times I)$ of $E(\A)$ maps into the image of the isometry $\F_I \colon L^2(X;\ell^2(I)) \to L^2(Y\times I)$, with
\[ T = \F_I \int_X^{\oplus} \tilde{T}(x)\, d\mu(x). \]
Here, $\tilde{T}(x) \colon J(x) \to \ell^2(I)$ is the analysis operator of $\{ \varphi_i(x) \}_{i \in I}$ for a.e.\ $x\in X$.
\end{theorem}

The following lemma ensures that the operator $\int_X^\oplus \tilde{T}(x)\, d\mu(x)$ is well defined.
\begin{lemma} \label{lem:PtAnDec}
Let $J'$ be the $\ell^2(I)$-valued range function with $J'(x) = \ell^2(I)$ for all $x \in X$. With circumstances and notation as in Theorem~\ref{thm:MIAn}, $\tilde{T} \colon J \to J'$ is a bounded, measurable range operator.
\end{lemma}

\begin{proof}
Boundedness follows immediately from Theorem~\ref{thm:MIF}. For measurability, take any $u \in \H$ and $v = \{ v_i \}_{i\in I} \in \ell^2(I)$, and observe that
\[ \langle \tilde{T}(x) P_J(x) u, v \rangle = \sum_{i\in I} \langle P_J(x) u, \varphi_i(x) \rangle \overline{v_i} = \sum_{i\in I} \langle u, P_J(x) \varphi_i(x) \rangle \overline{v_i} = \sum_{i\in I} \langle u, \varphi_i(x) \rangle \overline{v_i} \]
for a.e.\ $x \in X$. Since each $\varphi_i \colon X \to \H$ is measurable, Pettis's Measurability Theorem ensures that each mapping $x \mapsto \langle u, \varphi_i(x) \rangle \overline{v_i}$ is measurable on $X$. Taking the sum, we find that $x \mapsto \langle \tilde{T}(x) P_J(x) u, v \rangle$ is measurable.
\end{proof}

\begin{proof}[Proof of Theorem~\ref{thm:MIAn}]
If $\psi \in S(\A)$ and if~\eqref{eq:PFT} applies to $\varphi := \int_X^\oplus \tilde{T}(x)\, d\mu(x) \psi$, then we easily compute
\[ \Bigl[ \mathcal{F}_I \int_X^\oplus \tilde{T}(x)\, d\mu(x) \psi \Bigr](t,i) = \int_X \bigl[ \tilde{T}(x) \psi(x) \bigr]_i \overline{g_t(x)}\, d\mu(x) = \int_X \langle \psi(x), \varphi_i(x) \rangle \overline{g_t(x)}\, d\mu(x) \]
\[ = \langle \psi, M_{g_t} \varphi_i \rangle = (T\psi)(t,i). \]
It remains to show that~\eqref{eq:PFT} applies for all $\psi$ in a spanning set for a dense subspace.

To that end, take any $u \in \H$ and any measurable subset $E\subset X$ with $\mu(E) < \infty$, and define $\psi \colon X \to \H$ by $\psi(x) = \chi_E(x) P_J(x) u$. It is straightforward to show that $\psi \in L^2(X;\H)$, and indeed, $\psi \in S(\A)$. If $B > 0$ is a Bessel bound for $E(\A)$, then $\{ \varphi_i(x) \}_{i \in I}$ is $B$-Bessel for a.e.\ $x \in X$, by Theorem~\ref{thm:MIF}. Consequently,
\[ \sum_{i\in I} \int_X | \langle \psi(x), \varphi_i(x) \rangle |^2 d\mu(x) = \int_X \chi_E(x) \sum_{i\in I} | \langle P_J(x) u, \varphi_i(x) \rangle |^2 d\mu(x) \leq \mu(E) B \Norm{u}^2 < \infty, \]
so that each of the mappings $x \mapsto \langle \psi(x), \varphi_i(x) \rangle$ belongs to $L^2(X)$. A simple application of Cauchy-Schwarz shows it is also in $L^1(X)$. Therefore,~\eqref{eq:PFT} applies to $\int_X^\oplus \tilde{T}(x)\, d\mu(x) \psi$, and the above calculation shows that $T\psi = \mathcal{F}_I \int_X^\oplus \tilde{T}(x)\, d\mu(x) \psi$.

Finally, $\psi$ is the projection onto $S(\A)$ of the pure tensor $\chi_E u$, by Lemma~\ref{lem:rangProj}, and such pure tensors span a dense subspace of $L^2(X;\H) \cong L^2(X) \otimes \H$. Hence, $T\psi = \mathcal{F}_I \int_X^\oplus \tilde{T}(x)\, d\mu(x) \psi$ for all $\psi$ in a dense subspace of $S(\A)$. By continuity, the same holds for all $\psi \in S(\A)$.
\end{proof}

As a corollary, we obtain a measure-theoretic generalization of~\cite[Theorem~5.1]{B}.

\begin{cor} \label{cor:MIFrmOp}
Under the circumstances of Theorem~\ref{thm:MIAn}, the frame operator $S \colon S(\A) \to S(\A)$ of $E(\A)$ is multiplication invariant, with
\[ S = \int_X^{\oplus} S(x)\, d\mu(x). \]
Here, $S(x) \colon J(x) \to J(x)$ is the frame operator of $\{ \varphi_i(x)\}_{i\in I}$ for a.e.\ $x \in X$.
\end{cor}

\begin{proof}
This is immediate from Theorem~\ref{thm:MIAn}, Theorem \ref{prop}(iv), and the fact that $\F_I$ is an isometry.
\end{proof}

\begin{cor} \label{cor:MIGram}
Under the circumstances of Theorem~\ref{thm:MIAn}, the Gramian of $E(\A)$ conjugates to an MI operator on $L^2(X;\ell^2(I))$, with
\begin{equation} \label{eq:MIGrm}
\F_I^* (T T^*) \F_I = \int_X^\oplus \tilde{T}(x) \tilde{T}(x)^*\, d\mu(x).
\end{equation}
\end{cor}

\begin{proof}
This is immediate from Theorem~\ref{thm:MIAn} and Theorem \ref{prop}(iv), since $\F_I$ is an isometry whose range contains that of $T$.
\end{proof}

\subsection{Other properties of Bessel systems}

We now explore several other properties of a Bessel system. For instance, the following relates to the ``disjoint'' or ``orthogonal'' frames independently introduced by Balan~\cite{Ba} and by Han and Larson~\cite{HL}. See~\cite{GH} for a detailed study of disjoint continuous frames. The following measure-theoretic result has many predecessors in the setting of LCA groups, including~\cite{W04,KKLS07,GS18}.

\begin{cor} \label{cor:MIOrth}
Let $\A=\{\varphi_i\}_{i\in I}$ and $\A'=\{\varphi'_i\}_{i\in I}$ be sequences in $L^2(X;\H)$ for which $E(\A)$ and $E(\A')$ are both Bessel. Then the following are equivalent:
\begin{enumerate}[(i)]
\item The analysis operators of $E(\A)$ and $E(\A')$ have orthogonal ranges.
\item For a.e.\ $x\in X$, the analysis operators of $\{\varphi_i(x)\}_{i \in I}$ and $\{\varphi_i'(x)\}_{i \in I}$ have orthogonal ranges.
\end{enumerate}
\end{cor}

\begin{proof}
Let $T\colon S(\A) \to L^2(Y \times I)$ and $T' \colon S(\A') \to L^2(Y\times I)$ be the analysis operators of $E(\A)$ and $E(\A')$, respectively. Using notation as in Theorem~\ref{thm:MIAn}, we have
\begin{equation} \label{eq:MIMG}
T^* T' = \int_X \tilde{T}(x)^*\, d\mu(x)\, \F_I^* \F_I \int_X \tilde{T}'(x)\, d\mu(x) =  \int_X \tilde{T}(x)^* \tilde{T}'(x)\, d\mu(x).
\end{equation}
The operators $T$ and $T'$ have orthogonal range if and only if $T^*T' = 0$. The equation above shows that that happens if and only if $\tilde{T}(x)^* \tilde{T}'(x) = 0$ for a.e.\ $x\in X$, if and only if $\tilde{T}(x)$ and $\tilde{T}'(x)$ have orthogonal ranges for a.e.\ $x\in X$.
\end{proof}

\begin{defn}
If $(\M,m)$ is a $\sigma$-finite measure space and $\K$ is a separable Hilbert space, two Bessel systems $\{ u_s \}_{s \in \M}$ and $\{ u_s' \}_{s \in \M}$ are called \emph{dual} when their respective analysis operators $T,T'\colon \K \to L^2(\M,m)$ satisfy $T^* T' = I$. In that case, both systems are frames. The \emph{canonical dual} of a frame $\{ u_s \}_{s\in \M}$ with frame operator $S$ is $\{ S^{-1} u_s \}_{s\in \M}$.
\end{defn}

\begin{cor} \label{cor:MIDual}
Let $\A=\{\varphi_i\}_{i\in I}$ and $\A'=\{\varphi'_i\}_{i\in I}$ be sequences in $L^2(X;\H)$ with $S(\A) = S(\A')$, and suppose that $E(\A)$ and $E(\A')$ are both frames for $S(\A)$. Then the following are equivalent:
\begin{enumerate}[(i)]
\item $E(\A)$ and $E(\A')$ are dual frames.
\item For a.e.\ $x\in X$, $\{\varphi_i(x)\}_{x\in X}$ and $\{\varphi_i'(x)\}_{x\in X}$ are dual frames.
\end{enumerate}
\end{cor}

\begin{proof}
This follows immediately from \eqref{eq:MIMG}.
\end{proof}

\begin{cor} \label{cor:MICan}
Let $\A = \{\varphi_i\}_{i\in I}$ be a sequence in $L^2(X;\H)$, and assume that $E(\A)$ is a frame for $S(\A)$. Let $J$ be the range function associated with $S(\A)$, and let $S(x) \colon J(x) \to J(x)$ be the frame operator of $\{\varphi_i(x)\}_{i\in I}$ a.e.\ $x\in X$. Then:
\begin{enumerate}[(i)]
\item The canonical dual frame for $E(\A)$ is $E(\A')$, where $\A' = \{ S^{-1} \varphi_i \}_{i\in I}$ is given by
\begin{equation} \label{eq:MICD}
(S^{-1} \varphi_i)(x) = S(x)^{-1} \varphi_i(x) \qquad (\text{a.e.\ }x\in X).
\end{equation}
\item The canonical tight frame for $E(\A)$ is $E(\A'')$, where $\A'' = \{ S^{-1/2} \varphi_i \}_{i\in I}$ is given by
\begin{equation} \label{eq:MICT}
(S^{-1/2} \varphi_i)(x) = S(x)^{-1/2} \varphi_i(x) \qquad (\text{a.e.\ }x\in X).
\end{equation}
\end{enumerate}
\end{cor}
Part (i) generalizes \cite[Theorem~5.2]{B} to the measure-theoretic setting. 

\begin{proof}
We prove (ii), the proof of (i) being similar. Combining Corollary~\ref{cor:MIFrmOp} and Theorem~\ref{ft}, we see that $S^{-1/2}$ is the MI operator
\[ S^{-1/2} = \int_X^\oplus S(x)^{-1/2}\, d\mu(x). \]
For each $t \in Y$ and $i \in I$, we have $S^{-1/2} (M_{g_t} \varphi_i )= M_{g_t} (S^{-1/2} \varphi_i)$, so the canonical tight frame of $E(\A)$ is $S^{-1/2} E(\A) = E(\A'')$, where $\A'' = \{ S^{-1/2} \varphi_i \}_{i \in I}$ is as in \eqref{eq:MICT}.
\end{proof}

\section{Classification for Bessel systems of multiplications}
\label{sec:MIclass}

We continue to follow Assumptions~\ref{assumptions} and~\ref{assumptions 2}. The goal of this section is to classify Bessel systems $E(\{\varphi_i\}_{i\in I})$ in $L^2(X;\H)$.

\begin{defn}
Let $\M$ be a $\sigma$-finite measure space, and let $\K$ and $\K'$ be separable Hilbert spaces. Bessel systems $\{ u_s \}_{s\in \M}$ and $\{ u_s'\}_{s \in \M}$, in $\K$ and $\K'$ respectively, are called \emph{unitarily equivalent} if there is a unitary 
\[ U\colon \overline{\spn}\{u_s : s\in \M\} \to \overline{\spn}\{ u_s' : s\in \M\} \]
such that $U u_s = u_s'$ for every $s\in \M$. 
\end{defn}

\begin{theorem} \label{thm:MIE}
Let $\A=\{\varphi_i\}_{i\in I}$ and $\A' = \{ \varphi'_i\}_{i\in I}$ be sequences in $L^2(X;\H)$ indexed by $I$, and let $U\colon S(\A) \to S(\A')$ be a bounded operator. Then the following are equivalent:
\begin{enumerate}[(i)]
\item $U(M_{g_t} \varphi_i) = M_{g_t} \varphi'_i$ for every $t\in Y$ and every $i\in I$.
\item $U$ is an MI operator
\[ U = \int_X^\oplus U(x)\, d\mu(x) \]
whose fibers satisfy $U(x)\varphi_i(x) = \varphi_i'(x)$ a.e.\ on $X$.
\end{enumerate}
\end{theorem}

\begin{proof}
We prove (i)$\implies$(ii), the converse being trivial. Fix $i\in I$ and $\psi \in S(\A')$. Given any $W = \sum_{j=1}^n c_j M_{g_{t_j}}$ in $\spn\{ M_g : g \in \D\}$, we compute
\[ \langle W \varphi_i, U^* \psi \rangle = \langle UW \varphi_i, \psi \rangle = \Bigl\langle \sum_{j=1}^n c_j U M_{g_{t_j}} \varphi_i, \psi \Bigr\rangle = \Bigl\langle \sum_{j=1}^n c_j M_{g_{t_j}} \varphi_i', \psi \Bigr\rangle = \langle W \varphi_i', \psi \rangle. \]
Given any $f \in L^\infty(X)$, Lemma~\ref{lem:WSD} and Lemma~\ref{lem:WOTWS} provide a sequence $\{W_n\}_{n=1}^\infty$ in $\spn\{ M_g : g \in \D\}$ with $W_n \to M_f$ in the weak operator topology of $L^2(X;\H)$. Taking limits in the equation $\langle W_n \varphi_i, U^* \psi \rangle = \langle W_n \varphi_i', \psi \rangle$ then produces $\langle UM_f \varphi_i, \psi \rangle = \langle M_f \varphi_i', \psi \rangle$. As $\psi \in S(\A')$ was arbitrary, we conclude that $U M_f \varphi_i = M_f \varphi_i'$ for every $f \in L^\infty(X)$. Setting $f = 1$ shows that $U \varphi_i = \varphi_i'$. Moreover, for any $f,g\in L^\infty(X)$ we have
\[ M_f U M_g \varphi_i = M_f M_g \varphi_i' = M_{fg} \varphi_i' = U M_{fg} \varphi_i = U M_f M_g \varphi_i. \]
Since functions of the form $M_g \varphi_i$ span a dense subspace of $S(\A)$, we conclude that $M_f U = U M_f$ for every $f \in L^\infty(X)$. The remainder of (ii) follows from Theorem~\ref{thm:MIOp} and the fact that $U\varphi_i = \varphi_i'$.
\end{proof}

\begin{cor} \label{cor:MIBE1}
Let $\A=\{\varphi_i\}_{i\in I}$ and $\A'=\{\varphi'_i\}_{i\in I}$ be sequences in $L^2(X;\H)$ indexed by $I$, and suppose that $E(\A)$ and $E(\A')$ are both Bessel. 
Then the following are equivalent:
\begin{enumerate}[(i)]
\item $E(\A)$ is unitarily equivalent to $E(\A')$.
\item For a.e.\ $x\in X$, $\{\varphi_i(x)\}_{i\in I}$ is unitarily equivalent to $\{\varphi'_i(x)\}_{i\in I}$.
\item For a.e.\ $x\in X$, $\{ \varphi_i(x)\}_{i\in I}$ and $\{ \varphi'_i(x) \}_{i\in I}$ have the same Gramian.
\end{enumerate}
\end{cor}

\begin{proof}
(i)$\implies$(ii).
Given $x\in X$, we denote $J(x) = \overline{\spn}\{\varphi_i(x) : i \in I\}$ and $J'(x) = \overline{\spn}\{\varphi_i'(x) : i \in I\}$. 
If there is a unitary $U \colon S(\A) \to S(\A')$ such that $UM_{g_t} \varphi_i = M_{g_t}\varphi_i'$ for every $t \in Y, i \in I$, then Theorem~\ref{thm:MIE} implies that $U = \int_X^\oplus U(x)\, d\mu(x)$ with $U(x) \varphi_i(x) = \varphi_i'(x)$ a.e.\ on $X$. By Theorem~\ref{prop}(v), $U(x) \colon J(x) \to J'(x)$ is unitary for a.e.\ $x\in X$.

(ii)$\implies$(iii) is obvious.

(iii)$\implies$(i).
Generally speaking, whenever $\K$ and $\K'$ are two Hilbert spaces with total collections of vectors $\{ u_i \}_{i\in A}$ and $\{ u_i' \}_{i\in A}$, respectively, it is easy to show that there is a unitary $U \colon \K \to \K'$ mapping $U(u_i) = u_i'$ for every $i\in A$ if and only if $\langle u_i, u_j \rangle = \langle u_i', u_j' \rangle$ for every $i,j \in A$.
In the present case, when (iii) holds it is easy to show that
\[ \langle M_{g_t} \varphi_i, M_{g_s} \varphi_j \rangle = \langle M_{g_t} \varphi'_i, M_{g_s} \varphi'_j \rangle \]
for every $s,t \in Y$ and $i,j \in I$, so there is a unitary $U : S(\A) \to S(\A')$ mapping $U( M_{g_t} \varphi_i ) = M_{g_t} \varphi_i'$ for every $t \in Y$ and $i \in I$, as desired.
\end{proof}

In the discrete setting, unitary equivalence classes of Bessel sequences indexed by $I$ are in one-to-one correspondence with positive operators on $\ell^2(I)$, which are their Gramians. Corollary~\ref{cor:MIBE1} indicates that, in our MI setting, we should consider bounded, measurable fields $\{\mathrm{Gr}(x)\}_{x\in X}$ of positive operators on $\ell^2(I)$. However, when $\mu(X) = \infty$, not every such field will correspond to a Bessel system $E(\A)$, since the square norm of a vector in $\A$ can be found by integrating a diagonal entry of $\{\mathrm{Gr}(x)\}_{x\in X}$.

\begin{defn}
We call a positive MI operator
\[ \int_X^\oplus \mathrm{Gr}(x)\, d\mu(x) \colon L^2(X;\ell^2(I)) \to L^2(X;\ell^2(I)) \]
\emph{integrable} if 
\[ \int_X \langle \mathrm{Gr}(x) \delta_i, \delta_i \rangle\, d\mu(x) < \infty \qquad \text{for every }i\in I. \]
Here, $\delta_i \in \ell^2(I)$ is the canonical basis element corresponding to $i \in I$.
\end{defn}

\begin{rem}
\label{rem:IntGram}
When $\int_X^\oplus \mathrm{Gr}(x)\, d\mu(x)$ is a positive, integrable MI operator on $L^2(X;\ell^2(I))$, the off-diagonal entries of $\{ \mathrm{Gr}(x) \}_{x\in X}$ also lie in $L^1(X)$.
Indeed, $\mathrm{Gr}(x) \geq 0$ a.e.~$x\in X$ by Theorem~\ref{spec}(i), and successive applications of Cauchy--Schwarz show that
\[
\int_X | \langle \mathrm{Gr}(x) \delta_j, \delta_i \rangle |\, d\mu(x)
=
\int_X | \langle \mathrm{Gr}(x)^{1/2} \delta_j, \mathrm{Gr}(x)^{1/2} \delta_i \rangle | \, d\mu(x)
\]
\[
\leq
\left\{ \int_X \langle \mathrm{Gr}(x) \delta_j, \delta_j \rangle \, d\mu(x) \right\}^{1/2}
\left\{ \int_X \langle \mathrm{Gr}(x) \delta_i, \delta_i \rangle \, d\mu(x) \right\}^{1/2}
< \infty.
\]

When $I$ is finite, it follows that a positive MI operator $\int_X^\oplus \mathrm{Gr}(x)\, d\mu(x)$ is integrable if and only if 
\[ \int_X | \langle \mathrm{Gr}(x) u, v \rangle |\, d\mu(x)< \infty \qquad \text{ for every }u,v \in \ell^2(I). \]
However, that is not the case when $I$ is infinite. 
For a counterexample, take $X = I = \N$, and let $\mathrm{Gr}(n)$ be orthogonal projection onto $\overline{\spn}\{ \delta_k : k \geq n\}$. 
The corresponding MI operator is integrable, but if we take $u = \bigl\{ ( k^{-1} - (k+1)^{-1} )^{1/2} \bigr\}_{k=1}^\infty \in \ell^2(\N)$, then 
\[ \sum_{n=1}^\infty \langle \mathrm{Gr}(n) u, u \rangle 
= \sum_{n=1}^\infty \sum_{k=n}^\infty \Bigl[ \frac{1}{k} - \frac{1}{k+1} \Bigr]
= \sum_{n=1}^\infty \frac{1}{n} = \infty. \]
\end{rem}

\begin{theorem} \label{thm:MIBE2}
Unitary equivalence classes of Bessel systems $E(\{\varphi_i\}_{i\in I})$ in $L^2(X;\H)$ are in one-to-one correspondence with positive, integrable MI operators
\[ \int_X^\oplus \mathrm{Gr}(x)\, d\mu(x) \colon {L^2(X;\ell^2(I))} \to {L^2(X;\ell^2(I))} \]
having $\rank \mathrm{Gr}(x) \leq \dim \H$ a.e.\ $x \in X$. In this correspondence, $\mathrm{Gr}(x)$ is the Gramian of $\{\varphi_i(x)\}_{i\in I}$ a.e.\ $x \in X$.
\end{theorem}

\begin{proof}
Let $\A = \{\varphi_i\}_{i\in I}$ be a sequence in $L^2(X;\H)$ for which $E(\A)$ is Bessel, and let $\mathrm{Gr}(x) \colon \ell^2(I) \to \ell^2(I)$ be the Gramian of $\{ \varphi_i(x) \}_{i\in I}$ a.e.\ $x \in X$. Then $x \mapsto \mathrm{Gr}(x)$ is a bounded, measurable range function, by Lemma~\ref{lem:PtAnDec}, and $\mathrm{Gr} = \int_X^\oplus \mathrm{Gr}(x)\, d\mu(x)$ is an MI operator. Moreover, $\mathrm{Gr} \geq 0$ by Theorem~\ref{spec}(i), and 
\[ \rank \mathrm{Gr}(x) = \dim \overline{\spn}\{ \varphi_i(x) : i \in I\} \leq \dim \H \qquad (\text{a.e.\ }x \in X). \]
Since $\langle \mathrm{Gr}(x) \delta_i, \delta_i \rangle = \Norm{\varphi_i(x)}^2$, we have
\[ \int_X \langle \mathrm{Gr}(x) \delta_i, \delta_i \rangle d\mu(x) = \int_X \Norm{\varphi_i(x)}^2 d\mu(x) = \Norm{\varphi_i}^2 < \infty \]
for every $i \in I$. In other words, $\mathrm{Gr}$ is integrable.
Finally, let $\A' = \{ \varphi_i' \}_{i\in I}$ be another sequence in $L^2(X;\H)$ for which $E(\A')$ is Bessel, with corresponding MI operator $\mathrm{Gr}' \colon L^2(X;\ell^2(I)) \to L^2(X;\ell^2(I))$. Then $\mathrm{Gr} = \mathrm{Gr}'$ if and only if $E(\A)$ is unitarily equivalent to $E(\A')$, by Corollary~\ref{cor:MIBE1}. Overall, $E(\A) \mapsto \mathrm{Gr}$ induces a well-defined, one-to-one mapping from unitary equivalence classes of Bessel systems $E(\{\varphi_i\}_{i\in I})$ to positive, integrable MI operators $\mathrm{Gr} = \int_X^\oplus \mathrm{Gr}(x)\, d\mu(x)$ with $\rank \mathrm{Gr}(x) \leq \dim \H$ a.e.\ $x \in X$. It remains to prove that this mapping is onto.

To that end, let $\mathrm{Gr}$ be a positive, integrable MI operator on $L^2(X;\ell^2(I))$ corresponding to the measurable range operator $x\mapsto \mathrm{Gr}(x)$, and assume that $\rank \mathrm{Gr}(x) \leq \dim \H$ a.e.\ $x \in X$.
Then $\mathrm{Gr}(x) \geq 0$ a.e.\ $x \in X$, by Theorem~\ref{spec}(i). Given $i \in I$, we define $\psi_i(x) = \mathrm{Gr}(x)^{1/2} \delta_i$. As above, $\psi_i \in L^2(X;\ell^2(I))$ since $\mathrm{Gr}$ is integrable. Moreover, $\langle \psi_j(x), \psi_i(x) \rangle = \langle \mathrm{Gr}(x) \delta_j, \delta_i \rangle$, so $\{ \psi_i(x) \}_{i \in I}$ has Gramian $\mathrm{Gr}(x)$ a.e.\ $x \in X$. To complete the proof, we need only transport $\{ \psi_i \}_{i\in I}$ to a corresponding sequence $\{ \varphi_i \}_{i\in I}$ in $L^2(X;\H)$.

By Lemma~\ref{rk}, $J(x) := \overline{ \mathrm{Gr}(x) \ell^2(I)}$ is a measurable range function, and we clearly have $\psi_i \in V_J$ for each $i \in I$. Fix an orthonormal basis $\{ e_n \}_{n=1}^{|I|}$ for $\ell^2(I)$, and define
\[ J'(x) = \overline{\spn}\{ e_n : 1 \leq n \leq \rank \mathrm{Gr}(x) \} \qquad (\text{a.e.\ }x \in X). \]
Each of the sets $E_n := \{ x \in X : \rank \mathrm{Gr}(x) \leq n \} = \{ x \in X : \dim J(x) \leq n \}$ is measurable, so for any $u,v \in \ell^2(I)$ the mapping
\[ x \mapsto \langle P_{J'}(x) u, v \rangle = \Bigl \langle \sum_{n=1}^{\rank \mathrm{Gr}(x)} \langle u, e_n \rangle e_n, v \Bigr \rangle = \sum_{n=1}^{|I|} \chi_{E_n}(x) \langle u, e_n \rangle \langle e_n, v \rangle \]
is measurable on $X$. Therefore $J'$ is also measurable, and Theorem~\ref{dimo} provides an MI unitary $T \colon V_{J} \to V_{J'}$. 

Next, define $d = \operatorname*{ess\, sup}_{x\in X} \rank \mathrm{Gr}(x)$ and let $M = \overline{\spn}\{ e_n : 1 \leq n \leq d \}$, so that $J'(x) \subset M$ a.e.\ $x \in X$. We have $d \leq \dim \H$ by assumption, so there is a partial isometry $\ell^2(I) \to \H$ with initial space $M$. This gives rise to an MI partial isometry $U \colon L^2(X;\ell^2(I)) \to L^2(X;\H)$ whose initial subspace contains $V_{J'}$. The composition $UT$ is thus an MI operator that maps $V_J$ isometrically into $L^2(X;\H)$. Setting $\varphi_i = UT \psi_i$, we have $\langle \varphi_j(x), \varphi_i(x) \rangle = \langle \psi_j(x), \psi_i(x) \rangle$, so $\{ \varphi_i(x) \}_{i\in I}$ again has Gramian $\mathrm{Gr}(x)$ a.e.\ $x \in X$. In particular, $E(\{ \varphi_i \}_{i\in I})$ is Bessel by Theorem~\ref{thm:MIF}. This completes the proof.
\end{proof}

\begin{defn}
Given a Hilbert space $\K$, a positive operator $T \in B(\K)$ is called \emph{locally invertible} if there exists $\delta > 0$ such that $\sigma(T) \subset \{0\} \cup [\delta,\infty)$.
\end{defn}

\begin{cor} \label{cor:MIFrCl}
Unitary equivalence classes of systems $E(\{\varphi_i\}_{i\in I})$ in $L^2(X;\H)$ that are frames for $S(\{\varphi_i\}_{i\in I})$ are in one-to-one correspondence with locally invertible, integrable MI operators 
\[ \int_X^\oplus \mathrm{Gr}(x)\, d\mu(x) \colon L^2(X;\ell^2(I)) \to L^2(X;\ell^2(I)) \]
having $\rank \mathrm{Gr}(x) \leq \dim \H$ a.e.\ $x \in X$. In this correspondence, $\mathrm{Gr}(x)$ is the Gramian of $\{\varphi_i(x)\}_{i\in I}$ a.e.\ $x \in X$.
\end{cor}

\begin{proof}
Let $\A = \{ \varphi_i(x) \}_{i\in I}$ be a sequence in $L^2(X;\H)$ for which $E(\A)$ is $B$-Bessel, and let $\mathrm{Gr}(x)$ be the Gramian of $\{ \varphi_i(x) \}_{i\in I}$ a.e.\ $x \in X$. By Theorem~\ref{thm:MIF}, $E(\A)$ is a frame for $S(\A)$ if and only there exists $A>0$ such that $\sigma(\mathrm{Gr}(x)) \subset \{0\} \cup [A,B]$ a.e.\ $x \in X$. The latter is equivalent to local invertibility of $\mathrm{Gr}$ by Theorem~\ref{spec}(ii). In light of Theorem~\ref{thm:MIBE2}, the proof is complete.
\end{proof}

\begin{defn}
A measurable range function 
\[ J \colon X \to \{\text{closed subspaces of }\ell^2(I)\} \]
is \emph{integrable} if the projection $\int_X^\oplus P_J(x)\, d\mu(x)$ onto $V_J$ is integrable as an MI operator. Equivalently, $J$ is integrable if
\[ \int_X \langle P_J(x) \delta_i, \delta_i \rangle\, d\mu(x) < \infty \qquad \text{for every }i \in I. \]
\end{defn}

\begin{rem}
Given a measurable range function $J \colon X \to \{\text{closed subspaces of }\ell^2(I)\}$, notice that
\[ \sum_{i\in I} \int_X \langle P_J(x) \delta_i, \delta_i \rangle\, d\mu(x) = \int_X \dim J(x)\, d\mu(x). \]
When $I$ is finite, it follows that $J$ is integrable if and only if its supporting set 
\[
\operatorname{supp} J := \{ x \in X : J(x) \neq \{0\} \}
\]
has finite measure. This is not the case when $I$ is infinite, however, as shown by the counterexample $X = I = \N$, $J(n) = \overline{\spn}\{ \delta_k : k \geq n \}$ (as in~Remark~\ref{rem:IntGram}). Here, $J$ is integrable and yet $\operatorname{supp} J = \mathbb{N}$.
\end{rem}

\begin{cor} \label{cor:MIPF}
Unitary equivalence classes of systems $E(\{\varphi_i\}_{i\in I})$ in $L^2(X;\H)$ that are Parseval frames for $S(\{\varphi_i\}_{i\in I})$ are in one-to-one correspondence with 
integrable range functions
\[ J \colon X \to \{ \text{closed subspaces of }\ell^2(I)\} \]
having $\dim J(x) \leq \dim \H$ a.e.\ $x\in X$. In this correspondence, the Gramian of $\{ \varphi_i(x) \}_{i\in I}$ is orthogonal projection onto $J(x)$ a.e.\ $x \in X$.
\end{cor}

\begin{proof}
Let $\A = \{ \varphi_i(x) \}_{i\in I}$ be a sequence in $L^2(X;\H)$ for which $E(\A)$ is $B$-Bessel, and let $\mathrm{Gr}(x)$ be the Gramian of $\{ \varphi_i(x) \}_{i\in I}$ a.e.\ $x \in X$. The argument given in the proof of Corollary~\ref{cor:MIFrCl} shows that $E(\A)$ is a Parseval frame for $S(\A)$ if and only if $\sigma(\mathrm{Gr}(x)) \subset \{0,1\}$ a.e.\ $x \in X$, if and only if $\mathrm{Gr}(x)$ is an orthogonal projection a.e.\ $x \in X$. This completes the proof with $J(x)$ equal to the range of $\mathrm{Gr}(x)$ a.e.\ $x \in X$, by Theorem~\ref{thm:MIBE2}.
\end{proof}

\section{Applications for admissible representations of abelian groups} \label{sec:LCAFrm}

Throughout this section, we fix a locally compact abelian group $G$, which we assume to be second countable. As usual, $\hat{G}$ denotes the Pontryagin dual group. Haar measures on $G$ and $\hat{G}$ are simply denoted $dx$ and $d\alpha$, and scaling is assumed to satisfy the Plancherel theorem. Our goal is to interpret the results of the previous sections in the special case where:
\begin{itemize}
\item
$X = \hat{G}$ with $d\mu(\alpha) = d\alpha$,
\item
$Y = G$ with $d\nu(x) = dx$, and
\item
$\D = \{ \hat{x} \}_{x\in G}$, where $\hat{x} \colon \hat{G} \to \mathbb{T}$ is given by $\hat{x}(\alpha) = \overline{\alpha(x)}$.
\end{itemize}
By Pontryagin duality, $\D$ is the set of characters on $\hat{G}$. It is a Parseval determining set for $L^1(\hat{G})$ by Lemma~5.2 in~\cite{I}. The corresponding Plancherel transform is the inverse Fourier transform $\mathcal{F} \colon L^2(\hat{G}) \to L^2(G)$, $\mathcal{F} f = \check{f}$.

\begin{defn}
A representation $\pi \colon G \to U(\H)$ is called \emph{admissible} if there is a sequence $\{ u_i \}_{i \in I}$ in $\H$ for which $\{ \pi(x) u_i \}_{x \in G,\, i \in I}$ is a complete Bessel system in $\H$.
\end{defn}

This usage appears in~\cite{TW06} for discrete $G$.
A related condition, ``square integrable''~\cite{R04}, was shown in~\cite{W08} to be equivalent to the above.
In~\cite{W08}, the term ``$\sigma$-admissible'' is used instead of the above, with ``admissible'' reserved for the case where $|I|=1$. 

Given an admissible representation $\pi \colon G \to U(\H)$, Theorem~\ref{thm:LCAMod} provides the existence of a separable Hilbert space $\mathcal{K}$ and a linear isometry $U \colon \H \to L^2(\hat{G};\K)$ intertwining $\pi$ with modulation. Since modulation is simply multiplication by $\D$, Theorem~\ref{thm:MISpa} furnishes a measurable range function $J \colon \hat{G} \to \{ \text{closed subspaces of }\K\}$ such that $U(\H) = V_J$. We take this as a jumping-off point for the following applications.

\subsection{Characterizations of admissibility}

\begin{defn}
Let $\pi \colon G \to U(\H)$ be a representation, where $\H$ is a separable Hilbert space. A \emph{direct integral decomposition} of $\pi$ consists of the following data:
\begin{itemize}
\item
a Borel measure $\mu$ on $\hat{G}$, 
\item
a separable Hilbert space $\K$, 
\item a measurable range function $J \colon \hat{G} \to \{ \text{closed subspaces of }\K\}$ satisfying $J(\alpha) \neq \{0\}$ $\mu$-a.e.\ $\alpha \in \hat{G}$, and 
\item
a unitary $U \colon \H \to V_J \subset L^2(\hat{G},\mu;\K)$ intertwining $\pi$ with modulation.
\end{itemize}
We call $\mu$ a \emph{decomposing measure} for $\pi$, while $m(\alpha) : = \dim J(\alpha)$ is the \emph{multiplicity function}.
\end{defn}

Intuitively, this data determines a unitary equivalence $\pi \cong \int_{\hat{G}}^\oplus \alpha^{ (m(\alpha) )}\, d\mu(\alpha)$, where $\alpha$ is viewed as a one-dimensional representation of $G$, and $\alpha^{( m(\alpha) )}$ is the direct sum of $m(\alpha)$ copies of $\alpha$. Every representation of $G$ on a separable Hilbert space admits such a decomposition. Moreover, two representations are unitarily equivalent if and only if their decomposing measures are mutually absolutely continuous and their multiplicity functions agree a.e.\ (with respect to either measure)~\cite{F,He2}.

When $G$ is discrete, the equivalence of (i) and (iv) below is essentially contained in~\cite{TW06}.

\begin{theorem}
The following are equivalent for every representation $\pi \colon G \to U(\H)$ with $\H$ separable:

\begin{enumerate}[(i)]
\item
$\pi$ is admissible.
\item
A decomposing measure of $\pi$ is absolutely continuous with respect to Haar measure on $\hat{G}$.
\item
There is a separable Hilbert space $\mathcal{K}$ and a linear isometry $U \colon \H \to L^2(\hat{G};\mathcal{K})$ that intertwines $\pi$ with modulation.
\item
There is a sequence $\A$ in $\H$ such that $E(\A)$ is a Parseval frame for $\H$.
\end{enumerate}
\end{theorem}

\begin{proof}
(i) $\implies$ (ii) follows from Theorem~\ref{thm:LCAMod} and Theorem~\ref{thm:MISpa}.

(ii) $\implies$ (iii).
Assume (ii) holds.
We claim there is a measurable range function
\[ J \colon \hat{G} \to \{ \text{closed subspaces of }\ell^2(\Z)\} \]
and a unitary $U \colon \H \to V_J \subset L^2(\hat{G};\ell^2(\Z))$ intertwining $\pi$ with modulation.

Let $\mu$ be a decomposing measure for $\pi$, absolutely continuous with respect to Haar measure on $\hat{G}$.
Then there is a separable Hilbert space $\mathcal{K}'$, a measurable range function
\[ J' \colon \hat{G} \to \{ \text{closed subspaces of }\mathcal{K}'\}, \]
and a unitary $U' \colon \H \to V_{J'} \subset L^2(\hat{G},\mu;\mathcal{K}')$ that intertwines $\pi$ with modulation.
Fix a linear isometry $W \colon \mathcal{K}' \to \ell^2(\Z)$, and define 
\[ J \colon \hat{G} \to \{ \text{closed subspaces of }\ell^2(\Z)\} \]
by $J(\alpha) = W J'(\alpha)$. Then $\dim J(\alpha) = \dim J'(\alpha)$ for every $\alpha \in \hat{G}$.

Since $\hat{G}$ is second countable, its Haar measure is $\sigma$-finite, while $\mu$ is $\sigma$-finite by Lemma~\ref{lem:SigFin}.
Therefore, there exists a Radon-Nikodym derivative $f \colon \hat{G} \to [0,\infty)$ to write $d\mu(\alpha) = f(\alpha) d\alpha$.
Setting 
\[ E := \{ \alpha \in \hat{G} : f(\alpha) \neq 0 \}, \]
one easily shows that $d\mu(\alpha)$ and $\chi_E(\alpha) d\alpha$ are mutually absolutely continuous.
By~\cite[Thm.~1, p.~10]{He2}, there is a unitary $U'' \colon V_{J'} \to V_J$ that intertwines modulations.
Defining $U = U''U'$ proves the claim.

(iii) $\implies$ (iv).
If (iii) holds, then Theorem~\ref{thm:MISpa} provides the existence of a measurable range function $J \colon \hat{G} \to \{ \text{closed subspaces of }\K\}$ such that $U(\H) = V_J$.
According to Theorem~2.6 of~\cite{BR}, there is a sequence $\{ \varphi_i \}_{i=1}^\infty$ in $L^\infty(\hat{G};\K)$ such that $\varphi_n(\alpha) \in J(\alpha)$ a.e.~$\alpha \in \hat{G}$, and such that $\{ \varphi_n(\alpha) \}_{n=1}^\infty$ forms a Parseval frame for $J(\alpha)$ a.e.~$\alpha \in \hat{G}$.
Write $\hat{G}$ as a disjoint union $\hat{G} = \bigcup_{m=1}^\infty E_m$ of sets with finite measure, and then define $\varphi_{m,n} = \chi_{E_m} \varphi_n$.
Then $\varphi_{m,n} \in V_J \subset L^2(\hat{G};\mathcal{K})$, and $\{ \varphi_{m,n}(\alpha) \}_{m,n=1}^\infty$ is a Parseval frame for $J(\alpha)$ a.e.~$\alpha\in \hat{G}$.
Finally, define $\A = \{ U^{-1} \varphi_{m,n} \}_{m,n=1}^\infty$. Then (iv) follows from Theorem~\ref{thm:MIF}.

(iv) $\implies$ (i) is trivial.
\end{proof}

\begin{rem}
Every admissible representation admits a Parseval $G$-frame, but the same cannot be said for \emph{equal norm} Parseval $G$-frames, as the following example shows.

Let $J \colon \hat{G} \to \{ \text{closed subspaces of }\K\}$ be a measurable range function, and consider the modulation representation on $V_J \subset L^2(\hat{G};\K)$. 
Suppose that $E(\A)$ is a complete $B$-Bessel system in $V_J$ for some $\A = \{ \varphi_i \}_{i\in I}$.
Using Theorem~\ref{thm:MIF}, one easily shows that 
\[ \sum_{i \in I} \Norm{ \varphi_i(\alpha) }^2 \leq B\cdot \dim J(\alpha) \qquad (\text{a.e.\ $\alpha \in \hat{G}$}), \]
and therefore
\[ \sum_{i \in I} \Norm{\varphi_i}^2 \leq B \int_{\hat{G}} \dim J(\alpha)\ d\alpha. \]
On the other hand, since $\{ \varphi_i(\alpha) \}_{i\in I}$ is complete in $J(\alpha)$ a.e.\ $\alpha \in \hat{G}$, we must have
\[ \operatorname*{ess\, sup}_{\alpha \in \hat{G}} \dim J(\alpha) \leq |I|. \]
When the dimension function is both integrable and essentially unbounded, we conclude that the generators cannot have equal norms.
Such representations exist even for $G=\Z$.
\end{rem}

\subsection{Characterizations of invariant subspaces and intertwining operators}

The remainder of this section is devoted to interpreting results from the MI setting for admissible representations of LCA groups.
In each case, the ``interpreter'' is an isometry of the form $U \colon \H \to L^2(\hat{G};\K)$ intertwining an admissible representation with modulation.
Its existence was established by Theorem~\ref{thm:LCAMod}.
All of our results follow immediately from corresponding statements in the MI setting.

The following are the LCA versions of Theorems~\ref{thm:MISpa} and~\ref{thm:MIOp}, respectively.

\begin{cor} \label{cor:piInv}
Let $\pi \colon G \to U(\H)$ be an admissible representation. Choose any linear isometry $U \colon \H \to L^2(\hat{G};\K)$ intertwining $\pi$ with modulation, where $\K$ is a separable Hilbert space, and let $J \colon \hat{G} \to \{ \text{closed subspaces of } \K \}$ be the measurable range function for which $U(\H) = V_J$. Then invariant subspaces of $\H$ are in one-to-one correspondence with equivalence classes of measurable range functions $J' \colon \hat{G} \to \{ \text{closed subspaces of } \K \}$ satisfying $J'(\alpha) \subset J(\alpha)$ a.e.~$\alpha \in \hat{G}$.
The invariant subspace corresponding to $J'$ is $U^{-1}(V_{J'})$. 
\end{cor}

\begin{cor} \label{cor:piInt}
Let $\pi \colon G \to U(\H)$ and $\pi' \colon G \to U(\H')$ be admissible representations. Choose linear isometries $U \colon \H \to L^2(\hat{G};\K)$ and $U' \colon \H' \to L^2(\hat{G};\K)$ intertwining $\pi$ and $\pi'$ with modulation, where $\K$ is a separable Hilbert space, and let $J$ and $J'$ be the measurable range functions for which $U(\H) = V_J$ and $U'(\H') = V_{J'}$. Then bounded operators $\H \to \H'$ that intertwine $\pi$ and $\pi'$ are in one-to-one correspondence with equivalence classes of range operators $R \colon J \to J'$.
The intertwining operator corresponding to $R$ is $(U')^* [\int_{\hat{G}}^\oplus R(\alpha)\, d\alpha] U$.
\end{cor}

\begin{example} \label{ex:fib}
Take $G= \Z^n$ acting on $\H = L^2(\R^n)$ by integer shifts $\pi(k)f = f (\cdot - k)$. We can identify $\hat{G} = [0,1)^n$ as a measure space. Then the fiberization operator $\mathcal{T} \colon L^2(\R^n) \to L^2([0,1)^n;\ell^2(\Z^n))$ given by
\[ (\mathcal{T} f)(x) = \left\{ \hat{f}(x - k) \right\}_{k\in \Z^n}  \qquad (f \in L^2(\R^n),\, x \in [0,1)^n) \]
 is a linear isometry that intertwines $\pi$ with modulation.
 Choosing $U = \mathcal{T}$ in Corollary~\ref{cor:piInv}, we recover the usual range-function characterization of shift-invariant spaces~\cite{B,He,S}.
 
 Similarly, if $\H,\H' \subset L^2(\R^n)$ are shift-invariant subspaces on which $G=\Z^n$ acts by integer shifts $\pi,\pi'$, then we can take $U = \left. \mathcal{T} \right|_{\H}$ and $U' = \left. \mathcal{T} \right|_{\H'}$ in Corollary~\ref{cor:piInt} to recover the characterization of shift-preserving operators in Theorem~4.5 of~\cite{B}.
\end{example}

\subsection{Characterizations of $G$-systems}

The pedigree of the following corollary stretches to the seminal work of Benedetto and Li~\cite{BL93}, as demonstrated by Example~\ref{ex:BL} below. Similar results have been widely reported in the setting of translation-invariant spaces on abelian groups~\cite{BHP2,BL93,B,BR,CP,I,KR,RS}, as well as for representations of finite groups~\cite{VW2,VW3}, discrete groups~\cite{BHP,TW06}, and compact groups~\cite{I2}.
We have the LCA version of Theorem~\ref{thm:MIF}:

\begin{cor} \label{cor:piF}
Fix an admissible representation $\pi \colon G \to U(\H)$. Given a sequence $\A = \{ u_i \}_{i\in I}$ in $\H$, put $S(\A) = \overline{\spn}\{\pi(x) u_i : x \in G,\, i \in I\}$. Choose any linear isometry $U \colon \H \to L^2(\hat{G};\K)$ intertwining $\pi$ with modulation, where $\K$ is a separable Hilbert space, and let $J$ be a measurable range function for which $V_J = US(\A)$. 
Then the following are equivalent, for any $A,B>0$:
\begin{enumerate}[(i)]
\item 
The orbit $\{ \pi(x) u_i \}_{x\in G,\, i \in I}$ is an $A,B$-frame (resp.\ $B$-Bessel system) for $S(\A)$.
\item 
The system $\{ (U u_i)(\alpha) \}_{i\in I}$ is an $A,B$-frame (resp.\ $B$-Bessel system) for $J(\alpha)$ a.e.\ $\alpha \in \hat{G}$.
\end{enumerate}
\end{cor}

As a special case of Corollary~\ref{cor:piF} we obtain a generalization of~\cite[Theorem 3]{BL93}.

\begin{cor} \label{cor:piFs}
Fix an admissible representation $\pi \colon G \to U(\H)$. Given a vector $u \in \H$, put $S(u) = \overline{\spn}\{\pi(x) u : x \in G\}$. Choose any linear isometry $U \colon \H \to L^2(\hat{G};\K)$ intertwining $\pi$ with modulation, where $\K$ is a separable Hilbert space. Then the following are equivalent, for any $A,B > 0$:
\begin{enumerate}[(i)]
\item
The orbit $\{ \pi(x) u \}_{x\in G}$ is a frame for $S(u)$ with bounds $A,B$.
\item
For a.e.~$\alpha \in \hat{G}$, $\Norm{ (Uu)(\alpha) }^2 \in \{0\} \cup [A,B]$.
\end{enumerate}
\end{cor}

\begin{example} \label{ex:BL}
Let $G = \Z^n$, acting on $\H = L^2(\R^n)$ by integer shifts, as in Example~\ref{ex:fib}. For any $f \in L^2(\R^n)$ and $x \in [0,1)^n$ we have $\Norm{ (\mathcal{T} f)(x) }^2 = \sum_{k \in \Z^n} | \hat{f}(x+k) |^2$. Taking $U= \mathcal{T}$ and $u = f$ in Corollary~\ref{cor:piFs}, we see that $\{ f(\cdot - k) \}_{k\in \Z}$ is an $A,B$-frame for its closed linear span if and only if $\sum_{k \in \Z^n} | \hat{f}(x+k) |^2 \in \{0\} \cup [A,B]$ a.e.~$x \in [0,1)^n$. This is~\cite[Theorem 3]{BL93}.
\end{example}

\begin{defn}
If $\K$ is a separable Hilbert space, then a \emph{Riesz basis} for $\K$ is a sequence $\{ u_i \}_{i\in I}$ for which there are \emph{bounds} $A,B > 0$ such that, for any sequence $\{ c_i \}_{i\in I} \in \ell^2(I)$ having finite support,
\[ A \sum_{i\in I} | c_i |^2 \leq \Norm{ \sum_{i\in I} c_i u_i }^2 \leq B \sum_{i \in I} | c_i |^2. \]
\end{defn}

Riesz bases in the MI setting were studied by the second author in~\cite{I}.
The following result interprets~\cite[Theorem~2.3]{I} for LCA groups.

\begin{cor}
In addition to the standing assumptions, assume that $G$ is discrete and that Haar measure is normalized as counting measure. With the same notation and hypotheses of Corollary~\ref{cor:piF}, the following are equivalent for any $A,B>0$:
\begin{enumerate}[(i)]
\item 
The orbit $\{ \pi(x) u_i \}_{x\in G,\, i \in I}$ is a Riesz basis for $S(\A)$ with bounds $A,B$.
\item 
The system $\{ (U u_i)(\alpha) \}_{i\in I}$ is a Riesz basis for $J(\alpha)$ with bounds $A,B$ a.e.\ $\alpha \in \hat{G}$.
\end{enumerate}
\end{cor}

\subsection{Analysis and synthesis operators}

For Bessel systems generated by LCA groups, the analysis operator and its corollaries can be represented by direct integrals.
The following are the LCA versions of Theorem~\ref{thm:MIAn}, Corollary~\ref{cor:MIFrmOp}, and Corollary~\ref{cor:MIGram}, respectively.

\begin{cor} \label{cor:piGram}
Fix an admissible representation $\pi \colon G \to U(\H)$. Given a sequence $\A = \{ u_i \}_{i\in I}$ in $\H$, put $S(\A) = \overline{\spn}\{\pi(x) u_i : x \in G,\, i \in I\}$. Choose any linear isometry $U \colon \H \to L^2(\hat{G};\K)$ intertwining $\pi$ with modulation, where $\K$ is a separable Hilbert space, and let $J$ be a measurable range function for which $V_J = US(\A)$.
If the orbit $E(\A) = \{ \pi(x) u_i \}_{i\in I}$ is Bessel, then its analysis operator $T \colon S(\A) \to L^2(G\times I)$ is given by
\[ T = \mathcal{F}_I \left[ \int_{\hat{G}}^\oplus \tilde{T}(\alpha)\, d\alpha \right] \left.U\right|_{S(\A)}. \]
Here, $\tilde{T}(\alpha) \colon J(\alpha) \to \ell^2(I)$ is the analysis operator of $\{ (U u_i)(\alpha) \}_{i\in I}$ for a.e.~$\alpha \in \hat{G}$.
\end{cor}

\begin{cor}
Under the circumstances of Corollary~\ref{cor:piGram}, the frame operator $S \colon S(\A) \to S(\A)$ of $E(\A)$ is the conjugate of an MI operator, with
\[ S =  \left( \left.U\right|_{S(\A)} \right)^* \left[ \int_{\hat{G}}^\oplus S(\alpha)\, d\alpha \right] \left.U\right|_{S(\A)}. \]
Here, $S(\alpha) \colon J(\alpha) \to J(\alpha)$ is the frame operator of $\{ (U u_i)(\alpha) \}_{i\in I}$ for a.e.~$\alpha \in \hat{G}$.
\end{cor}

\begin{cor} 
Under the circumstances of Corollary~\ref{cor:piGram}, the Gramian of $E(\A)$ conjugates to an MI operator on $L^2(\hat{G};\ell^2(I))$, with
\[ \F_I^* (T T^*) \F_I = \int_{\hat{G}}^\oplus \tilde{T}(\alpha) \tilde{T}(\alpha)^*\, d\alpha. \]
\end{cor}

\subsection{Other properties of $G$-systems}

Other properties of Bessel systems, such as orthogonality and duality, can also be studied from the fibers.
The proof of the following is similar to that of Corollary~\ref{cor:MIOrth}.
We leave details to the reader.

\begin{cor}
Let $\pi \colon G \to U(\H)$ and $\pi' \colon G \to U(\H')$ be admissible representations. Fix sequences $\A = \{ u_i \}_{i\in I}$ and $\A' = \{ u_i' \}_{i\in I}$ in $\H$ and $\H'$, respectively, whose orbits $E(\A) = \{ \pi(x) u_i \}_{x\in G,\, i \in I}$ and $E(\A') = \{ \pi'(x) u_i' \}_{x\in G,\, i \in I}$ are both Bessel. Then the following are equivalent, for any choice of linear isometries $U \colon \H \to L^2(\hat{G};\K)$ and $U' \colon \H' \to L^2(\hat{G};\K')$ intertwining $\pi$ and $\pi'$ with modulation, respectively, where $\K$ and $\K'$ are separable Hilbert spaces:
\begin{enumerate}[(i)]
\item 
The analysis operators of $E(\A)$ and $E(\A')$ have orthogonal ranges.
\item
For a.e.~$\alpha \in \hat{G}$, the analysis operators of $\{ (U u_i)(\alpha) \}_{i \in I}$ and $\{ (U' u_i')(\alpha) \}_{i \in I}$ have orthogonal ranges.
\end{enumerate}
\end{cor}

We also have LCA versions of Corollaries~\ref{cor:MIDual} and~\ref{cor:MICan}, respectively.

\begin{cor}
Let $\pi \colon G \to U(\H)$ be an admissible representation.
Fix sequences $\A = \{ u_i \}_{i\in I}$ and $\A' = \{ u_i' \}_{i\in I}$ in $\H$ generating the same invariant subspace
\[ S(\A) = \overline{\spn}\{\pi(x) u_i : x \in G,\, i \in I\} = \overline{\spn}\{\pi(x) u_i' : x \in G,\, i \in I\} = S(\A'), \]
such that $E(\A) = \{ \pi(x) u_i \}_{x\in G,\, i \in I}$ and $E(\A') = \{ \pi(x) u_i' \}_{x\in G,\, i \in I}$ are both frames for $S(\A)$. 
Then the following are equivalent, for any choice of linear isometry $U \colon \H \to L^2(\hat{G};\K)$ intertwining $\pi$ with modulation, where $\K$ is a separable Hilbert space:
\begin{enumerate}[(i)]
\item
$E(\A)$ and $E(\A')$ are dual frames.
\item
For a.e.~$\alpha \in \hat{G}$, $\{ (U u_i)(\alpha) \}_{i\in I}$ and $\{ (U u_i')(\alpha) \}_{i\in I}$ are dual frames.
\end{enumerate}
\end{cor}

\begin{cor}
Let $\pi \colon G \to U(\H)$ be an admissible representation. Given a sequence $\A = \{ u_i \}_{i\in I}$, denote $E(\A) = \{ \pi(x) u_i \}_{i\in I}$ and $S(\A) = \overline{\spn}\{\pi(x) u_i : x \in G,\, i \in I\}$. Assume that $E(\A)$ is a frame for $S(\A)$, with frame operator $S$. Then:
\begin{enumerate}[(i)]
\item
The canonical dual frame for $E(\A)$ is $E(\A')$, where $\A' = \{ S^{-1} u_i \}_{i\in I}$.
\item
The canonical tight frame for $E(\A)$ is $E(\A'')$, where $\A'' = \{ S^{-1/2} u_i \}_{i\in I}$.
\end{enumerate}
\end{cor}

\subsection{Classifications of $G$-systems}

Finally, we interpret the results of Section~\ref{sec:MIclass} to classify Bessel systems generated by LCA groups.
When $N \in \{1,2,\dotsc \}$, we denote $\ell^2_N:= \ell^2(\{1,\dotsc,N\})$. Similarly, $\ell^2_\infty := \ell^2(\{1,2,\dotsc \})$.
The following is the LCA version of Theorem~\ref{thm:MIBE2} (resp.~Corollary~\ref{cor:MIFrCl}).

\begin{cor} \label{cor:piCl1}
For $N \in \{1,2,\dotsc,\infty\}$, unitary equivalence classes of Bessel $G$-systems (resp.\ $G$-frames) having $N$ generators are in one-to-one correspondence with positive (resp.\ locally invertible), integrable MI operators on $L^2(\hat{G};\ell^2_N)$.
\end{cor}

For Parseval frames, Corollary~\ref{cor:MIPF} says the following in the LCA setting.

\begin{cor} \label{cor:piCl2}
For $N \in \{1,2,\dotsc,\infty\}$, unitary equivalence classes of Parseval $G$-frames having $N$ generators are in one-to-one correspondence with 
integrable range functions
\[ J \colon \hat{G} \to \{ \text{closed subspaces of }\ell^2_N\}. \]
\end{cor}

\begin{rem} \label{rem:class}
We summarize the classifications above for the reader primarily interested in $G$-systems.
As in Theorem~\ref{thm:MIBE2}, the correspondences in Corollaries~\ref{cor:piCl1} and~\ref{cor:piCl2} pass through Gramians. 
Specifically, let $\pi \colon G \to U(\H)$ be an admissible representation, and let $\A = \{ u_i \}_{i \in I}$ be a sequence in $\H$ whose orbit $E(\A) = \{ \pi(x) u_i \}_{x\in G,\, i\in I}$ is a complete Bessel $G$-system. If $N = |I|$, then we can assume without loss of generality that $I = \{ 1,2,\dotsc,N\}$ (if $N < \infty$) or $I = \{ 1,2,\dotsc \}$ (if $N = \infty$). Choose any linear isometry $U \colon \H \to L^2(\hat{G};\ell^2_N)$ intertwining $\pi$ with modulation. For a.e.~$\alpha \in \hat{G}$, let $\mathrm{Gr}(\alpha) \colon \ell^2_N \to \ell^2_N$ be the Gramian of the Bessel system $\{ (U u_i)(\alpha) \}_{i\in I}$. Then $\mathrm{Gr} := \int_{\hat{G}}^\oplus \mathrm{Gr}(\alpha)\, d\alpha$ is the corresponding MI operator from Corollary~\ref{cor:piCl1}. If $E(\A)$ is a Parseval frame, then $\mathrm{Gr}(\alpha)$ is orthogonal projection onto a subspace $J(\alpha) \subset \ell^2_N$ for a.e.~$\alpha \in \hat{G}$. This defines the corresponding range function in Corollary~\ref{cor:piCl2}.

Conversely, for every positive, integrable MI operator $\mathrm{Gr} = \int_{\hat{G}}^\oplus \mathrm{Gr}(\alpha)\, d\alpha$ on $L^2(\hat{G};\ell^2_N)$, there exists a sequence $\A = \{ \varphi_i \}_{i=1}^N$ in $L^2(\hat{G};\ell^2_N)$ such that $\{ \varphi_i(\alpha) \}_{i=1}^N$ is Bessel with Gramian $\mathrm{Gr}(\alpha)$ for a.e.~$\alpha \in \hat{G}$. Then the orbit of $\A$ under the modulation representation is a Bessel $G$-system corresponding to $\mathrm{Gr}$.

In particular, consider an integrable range function $J \colon \hat{G} \to \{ \text{closed subspaces of }\ell^2_N\}$. For a.e.~$\alpha \in \hat{G}$, let $P(\alpha) \colon \ell^2_N \to \ell^2_N$ be orthogonal projection onto $J(\alpha)$. Define $\A = \{ \varphi_i \}_{i=1}^N$ to be the sequence in $L^2(\hat{G};\ell^2_N)$ given by $\varphi_i(\alpha) = P(\alpha) \delta_i$, where $\delta_i$ is a canonical basis element of $\ell^2_N$. Then the orbit of $\A$ under the modulation representation is a Parseval frame for $V_J$, and $J$ is the corresponding range function in Corollary~\ref{cor:piCl2}.
\end{rem}

\begin{example}
We interpret the classification for Parseval $G$-frames with a single generator as follows.
When $N=1$, we can identify $\ell^2_N$ with $\C$, so that $L^2(\hat{G};\ell^2_1) = L^2(\hat{G})$. 
A measurable range function $J \colon \hat{G} \to \{ \text{closed subspaces of $\C$}\}$ amounts to a measurable subset 
\[ E = \{ \alpha \in \hat{G} : J(\alpha) \neq \{0\} \} \subset \hat{G}, \]
and $J$ is integrable if and only if $|E| < \infty$. Overall, equivalence classes of Parseval $G$-frames with a single generator are in one-to-one correspondence with measurable subsets $E \subset \hat{G}$ (modulo null sets) satisfying $|E| < \infty$.

Given such an $E$, the construction in Remark~\ref{rem:class} produces the single generator $\varphi = \chi_E$. 
For each $x \in G$, let $\hat{x} \colon \hat{G} \to \T$ be the character $\hat{x}(\alpha) = \overline{\alpha(x)}$. Then the orbit of $\varphi$ under the modulation representation is $\Phi_E := \{ \hat{x} \cdot \chi_E \}_{x\in G}$. It is a Parseval frame for $L^2(E)$.

In the case where $G$ is finite, let $F \in \C^{\hat{G} \times G}$ be the discrete Fourier transform (DFT) matrix $F = \left[ \overline{ \alpha(x) } \right]_{\alpha \in \hat{G},\, x \in G}$. As a short, fat matrix, $\Phi_E$ is obtained by extracting the rows labeled by $E$ from $F$, and then scaling to account for the normalized Haar measure on $\hat{G}$. Explicitly, $\Phi_E = |G|^{-1/2} \left[ \overline{ \alpha(x) } \right]_{\alpha \in E,\, x \in G} \in \C^{E \times G}$. The \emph{harmonic frames} obtained in this way have been broadly studied and applied~\cite{CK,EB,GKK,GVT,HMRSU,SH}. We have recovered the identification of Parseval $G$-frames with harmonic frames from~\cite{VW2}. See~\cite{CW11} for a classification of harmonic frames under a broader notion of unitary equivalence that is independent of group structure.
\end{example}

\begin{rem}
Suppose that $G$ is a finite abelian group. Since there are only finitely many subsets $E \subset \hat{G}$, there are only finitely many unitarily inequivalent Parseval $G$-frames having a single generator~\cite{VW2}. On the other hand, when $N>1$ there are uncountably many integrable range functions $J \colon \hat{G} \to \{ \text{closed subspaces of } \ell^2_N \}$. Hence there are uncountably many unitarily inequivalent Parseval $G$-frames having $N$ generators, for each $N > 1$.
\end{rem}

\section{Applications for shift-invariant spaces}\label{S8}

In this section we give applications of MI operators in the study of shift-invariant (SI) spaces and SI operators acting on these spaces. This provides a unifying way of obtaining results on SI spaces and SI operators which was initiated in \cite{BR}. We shall focus on the setting of SI spaces invariant under translations by an abelian group, which were studied in \cite{I}. 

Suppose that $\mathcal G$ is a second countable locally compact group and $G \subset \mathcal G$ is an abelian subgroup. Given a function $f\colon \mathcal G \to \C$ and $y \in \mathcal G$, we will write $L_y f\colon \mathcal G \to \C$ for the left translation given by
\[
(L_y f)(x) = f(y^{-1}x) \qquad x\in \mathcal G.
\]
We equip $\mathcal G$ with a left invariant Haar measure. 

\begin{defn}
A closed subspace $ M \subset L^2(\mathcal G)$ is called $G$-translation-invariant, or $G$-TI,  if $L_y(M) = M$ for all $y\in G$. We say that an operator $T: M \to M'$ between two $G$-TI spaces $M$ and $M'$ is translation-invariant if
\[
T L_y = L_y T \qquad\text{for all }y\in G.
\]
\end{defn}

In the classical setting of $\mathcal G = \R^n$ and $G=\Z^n$ these are known as shift-invariant subspaces of $L^2(\R^n)$, whose study was pioneered by de Boor, DeVore, and Ron \cite{BDR2, BDR} and the first author \cite{B}. When $G$ is a discrete co-compact subgroup of a second countable locally compact abelian (LCA) group $\mathcal G$, shift-invariant spaces were studied by Cabrelli and Paternostro \cite{CP}. The assumptions on $G$ can be further relaxed to $G$ being merely co-compact \cite{BR} or an arbitrary subgroup of a second countable LCA group \cite{I}. Finally, the second author has extended results on TI spaces to the setting when $G$ is an abelian subgroup of a second countable locally compact group $\mathcal G$, which is not necessarily abelian. The study of translation invariant spaces can be reduced to the study of MI spaces using either fiberization operator or a generalized Zak transform which converts left translation operators $L_y$ into multiplication operators acting on appropriate vector-valued MI spaces $L^2(X;\mathcal H)$. The following two results make this process possible.

Theorem \ref{lcag} establishes the existence of a fiberization mapping in the setting of TI spaces. It was shown in various degree of generality by several authors \cite{BR, CP, KR}. The setting of an arbitrary subgroup of an LCA group is due to the second author \cite[Proposition 4.3]{I}. 

\begin{theorem}\label{lcag} Let $\mathcal G$ be a second countable locally compact {\bf abelian} group and let $G \subset \mathcal G$ be its subgroup. Let $\mu_{\mathcal G}$ and $\mu_{\hat{\mathcal G}}$ be Haar measures on $\mathcal G$ and $\hat{\mathcal G}$ such that the Plancherel formula holds. Let $\mu_{G^*}$ be a Haar measure on an annihilator $G^*$ of $G$ in $\hat {\mathcal G}$, which is given by
\[
G^* = \{ \gamma \in \hat{\mathcal G}: \gamma(x) =1 \quad\text{for all }x\in G\}.
\]

For a fixed Borel section $\Omega \subset \hat {\mathcal G}$ of $\hat {\mathcal G}/ G^*$, such that a bijection $\hat {\mathcal G}/ G^* \to \Omega$ sends compact sets to precompact sets, let $\mu_\Omega$ be 
the unique regular Borel measure on $\Omega$ satisfying the Weil identity
\[
\int_{\hat {\mathcal G}} f d\mu_{\hat {\mathcal G}} = \int_{\Omega} \int_{G^*} f(\alpha\gamma) d\mu_{G^*}(\gamma) d\mu_\Omega(\alpha) \qquad\text{for all }f\in L^1(\hat {\mathcal G}, \mu_{\hat {\mathcal G}}).
\]
For any $x\in G$, define $\hat{x} \colon \Omega \to \T$ as an evaluation mapping $\hat{x}(\alpha)=\ov{\alpha(x)}$ for $\alpha\in \Omega \subset \hat {\mathcal G}$.

Then the following hold:
\begin{enumerate}[(i)]
\item There is a unitary map
\[
\mathcal T\colon L^2(\mathcal G, \mu_{\mathcal G}) \to L^2(\Omega,\mu_\Omega; L^2(G^*,\mu_{G^*}))
\]
given by 
\[
(\mathcal T f)(\alpha)(\gamma) = \hat f( \alpha \gamma) \qquad \gamma \in G^*, \alpha \in \Omega, f\in L^2({\mathcal G},\mu_{{\mathcal G}}).
\]
\item
The set $\{\hat{x}\}_{x \in G}$ is a Parseval determining set for $L^1(\Omega,\mu_\Omega)$.
\item For any $x\in G$ we have
\begin{equation}\label{lcag5}
(\mathcal T L_x f)(\alpha) = \hat{x}(\alpha) \mathcal T f(\alpha) \qquad \text{for a.e. }\alpha\in \Omega.
\end{equation}
\end{enumerate}
\end{theorem}

\begin{proof}
The existence of a Borel section is a classical result \cite{FG} and so is the Weil formula \cite[Theorem 2.49]{F}. Parts (i) and (iii) were shown in \cite[Proposition 4.3]{I}. Indeed, \eqref{lcag5} follows from the identification of $\Omega$ with $\hat {\mathcal G}/G^*$, which in turn is isomorphic with $\hat G$. For any $x\in G$, define $\hat{x}' \colon \hat G \to \T$ as an evaluation mapping $\hat{x}(\alpha)=\ov{\alpha(x)}$ for $\alpha \in  \hat G$. It was observed in Section \ref{sec:LCAFrm} that the collection  $\{\hat{x}'\}_{x\in G}$ is a Parseval determining set for $L^1(\hat G)$. Since $\hat{x}$ and $\hat{x}'$ can be identified, (ii) follows.
\end{proof}

In the case when $G$ is a co-compact subgroup of $\mathcal G$, the annihilator $G^*$ is a discrete subgroup of $\hat{\mathcal G}$. Then, it is customary to take $\mu_{G^*}$ to be a counting measure and hence $L^2(G^*,\mu_{G^*})$ is identified with $\ell^2(G^*)$. Hence, in this case Theorem \ref{lcag} reduces to \cite[Lemma 3.5 and Proposition 3.7]{BR}.

Theorem \ref{lcas} establishes the existence of the Zak transform in the setting of an arbitrary abelian subgroup of a second countable locally compact group which is due to the second author \cite[Theorem 4.3]{I}.

\begin{theorem}\label{lcas}
 Let $\mathcal G$ be a second countable locally compact group and let $G \subset \mathcal G$ be its {\bf abelian} subgroup. Let $\mu_{\mathcal G}$ be a left Haar measure on $\mathcal G$. Let $\mu_{G}$ and $\mu_{\hat G}$ be Haar measures on $G$ and $\hat G$ such that the Plancherel formula holds.
 
For a fixed Borel section $\Delta \subset \mathcal G$ of $G\backslash {\mathcal G}$, such that a bijection $  G \backslash  {\mathcal G} \to \Delta$ sends compact sets to precompact sets, let $\mu_\Delta$ be 
the unique regular Borel measure on $\Delta$ satisfying the Weil identity
\begin{equation}\label{lcas1}
\int_{\mathcal G} f d\mu_{\mathcal G} = \int_{\Delta} \int_{G} f(x y) d\mu_{G}(x) d\mu_\Delta(y) \qquad\text{for all }f\in L^1(\mathcal G, \mu_{\mathcal G}).
\end{equation}
For any $x\in G$, define $\hat{x} \colon \hat G \to \T$ as an evaluation mapping $\hat{x}(\alpha)=\ov{\alpha(x)}$ for $\alpha \in  \hat G$.
Then the following holds:
\begin{enumerate}[(i)]
\item There is a unitary map
\[
\mathcal Z\colon L^2(\mathcal G, \mu_{\mathcal G}) \to L^2(\hat G, \mu_{\hat G}; L^2(\Delta,\mu_\Delta))
\]
given by 
\begin{equation}\label{lcas4}
(\mathcal Z f)(\alpha)(y) = \widehat{ f( \cdot y)}(\alpha) \qquad y\in \Delta, \alpha \in \hat G, f\in L^2({\mathcal G}, \mu_{\mathcal G}).
\end{equation}
\item
The set $\{\hat{x}\}_{x \in G}$ is a Parseval determining set of $L^1(\hat G,\mu_{\hat G})$.
\item
For any $x\in G$ we have
\begin{equation}\label{lcas5}
(\mathcal Z L_x f)(\alpha) = \hat{x} (\alpha) \mathcal T f(\alpha) \qquad \text{for a.e. }\alpha \in \hat G.
\end{equation}
\end{enumerate}
\end{theorem}

\begin{proof}
The existence of a Borel section again follows from the classical result \cite{FG}. However, the Weil formula involving right cosets $G\backslash {\mathcal G}$ is a non-trivial fact which was shown in \cite[Theorem 3.4]{I}. It is implicitly understood in \eqref{lcas1} that $x \mapsto f(xy)$ belongs to $L^1(G)$ for $\mu_\Delta$-a.e. $y\in \Delta$, and the function $y \mapsto \int_G f(xy) d\mu_G(x)$ is in $L^1(\Delta)$. In particular, the Fourier transform of $x \mapsto f(xy)$ is well-defined and denoted by $\widehat{ f( \cdot y)}$ in \eqref{lcas4}.
Parts (i) and (iii) were shown in \cite[Theorem 4.1]{I}. Finally, we have already seen that the collection  $\{\hat{x}\}_{x\in G}$ is a Parseval determining set for $L^1(\hat G)$. 
\end{proof}

Theorem \ref{lcag} and \ref{lcas} enable us to reduce the study of $G$-TI subspaces of $L^2(\mathcal G)$ to the framework of MI spaces where Theorem \ref{thm:MISpa} is readily available. Hence, we have the following result \cite[Theorem 5.1]{I}, which unifies a number of earlier results on the structure of SI spaces \cite{B, BR, BDR, KR}.

\begin{theorem}\label{ti}
 Let $\mathcal G$ be a second countable locally compact group and let $G \subset \mathcal G$ be its {\bf abelian} subgroup. Let $M \subset L^2(\mathcal G)$ be a $G$-TI space.
Equip $X=\hat G$ with Haar measure and let $\Delta$ be a Borel section of $G \backslash \mathcal G$ as in Theorem \ref{lcas}. 
\begin{enumerate}[(i)]
\item
There exists a measurable range function 
$J: X \to \{\text{closed subspaces of }L^2(\Delta)\}$, such that
\begin{equation}\label{ti1}
M= \{f \in L^2(X): \mathcal Z f(\alpha) \in J(\alpha) \text{ for a.e. }\alpha\in X \}.
\end{equation}

\item
Conversely, for any measurable range function $J$ as in (i), the space $M$ given by \eqref{ti1} is $G$-TI, and this correspondence is one-to-one provided we identify range functions that agree a.e. 

\item
If $\{f_i\}_{i\in I} \subset M$ is a countable family of generators of $M$ indexed by $I$,
\[
M = \ov{\spn}\{L_y f_i : y\in G, i\in I\},
\]
then
\[
J(\alpha) = \ov{\spn}\{\mathcal Z f(\alpha): i\in I\} \qquad\text{for a.e. }\alpha\in X.
\]

\item
Let $0<A\le B<\infty$. A collection 
\[
E^G(\{f_i\}_{i\in I} )=\{L_y f_i : y\in G, i\in I\}
\]
is a $A,B$-frame of $M$ if and only if $\{\mathcal Zf_i(\alpha): i\in I \}$ is $A,B$-frame for a.e. $\alpha\in X$.
\item 
Let $\{f_i\}_{i\in I}$ and $\{f'_i\}_{i\in I}$ be two countable families of functions in $M$ such that $E^G(\{f_i\}_{i\in I})$ and $E^G(\{f'_i\}_{i\in I})$ are Bessel sequences. The systems $E^G(\{f_i\}_{i\in I})$ and $E^G(\{f'_i\}_{i\in I})$ are dual frames in $M$ if and only if  $\{\mathcal Zf_i(\alpha): I \in I\}$ and $\{\mathcal Zf'_i(\alpha): I \in I \}$ are dual frames in $J(\alpha)$ for a.e. $\alpha\in X$. 
\item The above characterization also holds for canonical dual frames.
\end{enumerate}

In addition, if $\mathcal G$ is abelian, then all of the preceding results hold when $X= \Omega$ is a Borel section of $\hat {\mathcal G}/ G^*$ as in Theorem \ref{lcag}, $\Delta=G^*$, and the Zak transform $\mathcal Z$ is replaced by the fiberization transform $\mathcal T$.
\end{theorem}

\begin{proof}
Parts (i), (ii), and (iii) follow by combining Theorems \ref{thm:MISpa} and \ref{lcas}. Indeed, by Theorem \ref{lcas} there is one-to-one correspondence between multiplication-invariant spaces $V$ of $L^2(X;L^2(\Delta))$ with respect to $\{\hat{x}\}_{x\in G}$ and $G$-TI subspaces $M$ of $L^2(\mathcal G)$ given by $V=\mathcal Z(M)$. Hence, Theorem \ref{thm:MISpa} can be applied to $\Delta$ to yield a range function characterization of $V$. Parts (iv), (v), and (vi) then follow by Theorem \ref{thm:MIF}, Corollary \ref{cor:MIDual}, and Corollary \ref{cor:MICan}, resp.
\end{proof}

Theorem \ref{ti} enables us to define the dimension function for a $G$-TI space $M$ as 
\begin{equation}\label{dim5}
\dim^G_M\colon X \to \N \cup \{0,\infty\}, \qquad \dim^G_M(\alpha) = \dim J(\alpha), \ \alpha\in X.
\end{equation}
Here, $X=\hat G$ or in the case of abelian $\mathcal{G}$ we can take for $X$ a Borel section $\Omega$ of $\hat {\mathcal G}/ G^*$, which is isomorphic with $\hat G$. Note that the dimension functions defined via fiberization transform $\mathcal T$ or the Zak transform $\mathcal Z$ coincide regardless whether we take $\mathcal T$ or $\mathcal Z$ in Theorem \ref{ti}. This follows by the relation between $\mathcal T$ and $\mathcal Z$ shown by the second author, see the calculation after \cite[Proposition 4.3]{I}. The superscript $G$ in \eqref{dim5} indicates the dependence of the dimension function on a subgroup $G \subset \mathcal G$. Indeed, every $G$-TI space is obviously also $G'$-TI for any subgroup $G' \subset G$. However, when not necessary we shall drop this dependence by writing $\dim_M$.

In parallel to Theorem \ref{ti}, the study of TI operators can be reduced to the framework of MI operators where Theorem \ref{thm:MIOp} becomes very useful.

\begin{theorem}\label{tio}
 Let $\mathcal G$, $G \subset \mathcal G$, $X$ and $\Delta$ be as in Theorem \ref{ti}. Let $M, M' \subset L^2(\mathcal G)$ be two $G$-TI space with the corresponding range functions $J,J'\colon X \to \{\text{closed subspaces of }L^2(\Delta)\}$. Let $T\colon M \to M'$ be $G$-TI operator.
\begin{enumerate}[(i)]
\item 
There exists a bounded measurable range operator $R\colon J \to J'$ such that
\begin{equation}\label{tio1}
\mathcal Z \circ T \circ \mathcal Z^{-1} = \int^{\oplus}_X R(\alpha) d\mu(\alpha),
\end{equation}
where $\mu$ is a Haar measure on $X=\hat G$.
\item Conversely, for any measurable range operator $R\colon J \to J'$, an operator $T\colon M \to M'$ given by \eqref{tio1} is $G$-TI, and this correspondence is one-to-one provided we identify range operators that agree a.e. 
\item Any of the properties of $G$-TI operator $T$, which are listed in Theorem \ref{prop}, is reflected by the corresponding properties of a range operator $R$.
\item The spectra of range operator $R$ satisfy
\[
\sigma(R(\alpha)) \subset \sigma(T) \qquad\text{for $\mu$-a.e.\ }\alpha\in X.
\]
\item The relation \eqref{ft4} for a functional calculus between $T$ and $R$ holds when either $h$ is a holomorphic function on some neighborhood of $\sigma(T)$ or
$h$ is a bounded complex Borel function on $\sigma(T)$ and $T$ is normal.
\end{enumerate}
In addition, if $\mathcal G$ is abelian, then all of the preceding results hold when $X$ is a Borel section of $\hat {\mathcal G}/ G^*$, $\Delta=G^*$, and \eqref{tio1} replaced by
\begin{equation}\label{tio3}
\mathcal T \circ T \circ \mathcal T^{-1} = \int^{\oplus}_X R(\alpha) d\mu(\alpha),
\end{equation}
\end{theorem}

\begin{proof} Parts (i) and (ii) follow by combing Theorems \ref{thm:MIOp}, \ref{lcas}, and \ref{ti}(i)(ii).
Parts (iii), (iv), and (v) then follow by Theorem \ref{prop}, Corollary \ref{len}, and Theorem \ref{ft}, resp. 
\end{proof}

We conclude with the following result which is 
a generalization of \cite[Theorems 4.9 and 4.10]{B} to the setting of TI spaces. Theorem \ref{dimv} is an immediate consequence of Theorems \ref{dimo} and \ref{tio}.

\begin{theorem} \label{dimv}
Suppose that $M, M' \subset L^2(\mathcal G)$ are two $G$-TI spaces. The following are equivalent:
\begin{enumerate}[(i)]
\item $\dim_M(\alpha)=\dim_{M'}(\alpha)$ for a.e.\ $\alpha\in X$,
\item there exists a $G$-TI isometric isomorphism $T\colon M \to M'$,
\item there exists a $G$-TI isomorphism $T\colon M \to M'$.
\end{enumerate}
\end{theorem}

\appendix
\section{}

We have not been able to locate references for the results below, so we supply our own proofs.

\begin{lemma} \label{lem:WOTWS}
The weak operator topology on $L^\infty(X)$, viewed as an algebra of multiplication operators on $L^2(X;\H)$, is identical to the weak-$\ast$ topology it receives as the dual of $L^1(X)$, provided that $\H \neq \{0\}$.
\end{lemma}

\begin{proof}
Let $\{\phi_n\}_{n=1}^\infty$ be a sequence in $L^\infty(X)$, and let $\phi \in L^\infty(X)$. If $\phi_n \to \phi$ in the weak-$\ast$ topology, then for any $\varphi,\psi \in L^2(X;\H)$, we obviously have
\begin{align*}
\langle M_{\phi_n} \varphi, \psi \rangle &= \int_X \langle (M_{\phi_n} \varphi)(x), \psi(x) \rangle d\mu(x) = \int_X \langle \varphi(x), \psi(x) \rangle \phi_n(x)\, d\mu(x) \\
&\to \int_X \langle \varphi(x), \psi(x) \rangle \phi(x)\, d\mu(x) = \langle M_\phi \varphi, \psi \rangle,
\end{align*}
so $\phi_n \to \phi$ in the weak operator topology.

Conversely, when $\phi_n \to \phi$ in the weak operator topology and $f\in L^1(X)$ is arbitrary, we can find a measurable unimodular function $h\colon X \to \mathbb{T}$ such that $f(x) = h(x) |f(x)|$ for a.e.\ $x\in X$. Fix a unit vector $u \in \H$, and define $\varphi, \psi \colon X \to \H$ by $\varphi(x) = h(x) |f(x)|^{1/2}u$ and $\psi(x) = |f(x)|^{1/2} u$ for a.e.\ $x\in X$. Then
\[ \int_X \Norm{\varphi(x)}^2 d\mu(x) = \int_X \Norm{\psi(x)}^2 d\mu(x) = \int_X |f(x)|\, d\mu(x) < \infty, \]
and
\[ \langle M_{\phi_n} \varphi, \psi \rangle = \int_X \langle \phi_n(x)\cdot h(x) |f(x)|^{1/2} u , |f(x)|^{1/2} u \rangle d\mu(x) = \int_X \phi_n(x) f(x)\, d\mu(x) \]
converges to
\[ \langle M_\phi \varphi, \psi \rangle = \int_X \phi(x) f(x)\, d\mu(x), \]
so $\phi_n \to \phi$ in the weak-$\ast$ topology.
\end{proof}

\begin{lemma} \label{lem:SigFin}
If $G$ is a locally compact abelian group and $\H$ is a separable Hilbert space, then the decomposing measure of every representation $\pi \colon G \to U(\H)$ is $\sigma$-finite.
\end{lemma}

\begin{proof}
Let $\mu$ be a decomposing measure for $\pi$. By definition there exists a separable Hilbert space $\mathcal{K}$, a measurable range function
\[ J \colon \hat{G} \to \{ \text{closed subspaces of }\mathcal{K}\} \]
with $J(\alpha)\neq\{0\}$ $\mu$-a.e.\ $\alpha \in \hat{G}$, and a unitary $U \colon \H \to V_{J} \subset L^2(\hat{G},\mu;\mathcal{K})$ intertwining $\pi$ with modulation.

Since $V_{J} \cong \H$ is separable, it has a countable orthonormal basis $\{ \varphi_i \}_{i=1}^N$.
By Theorem \ref{thm:MISpa}, the range function $J$ satisfies 
\[ J(\alpha) = \overline{\spn}\{ \varphi_i(\alpha) : 1 \leq i \leq N\} \qquad (\alpha \in \hat{G}) \]
Since $J(\alpha) \neq \{0\}$ $\mu$-a.e.~$\alpha \in \hat{G}$, we have
\[ \mu(E_0) := \mu(\{ \alpha \in \hat{G} : \varphi_i(\alpha) = 0 \text{ for all }i \}) = 0. \]
Finally, $\hat{G}$ is the union of $E_0$ and the sets
\[ E_{i,n} := \{ \alpha \in \hat{G} : \Norm{\varphi_i(\alpha)}^2 > 1/n \}, \]
each of which satisfies $\mu(E_{i,n}) \leq n \Norm{\varphi_i}^2 < \infty$.
\end{proof}

\bibliographystyle{abbrv}
\bibliography{MI}

\end{document}